\documentclass[11pt]{amsart}
\usepackage{amssymb,amsmath}
\usepackage{color,comment}
\usepackage[nobysame]{amsrefs}

\def\Bbb{\mathbb}
\def\Cal{\mathcal}
\def\Dt{\partial_t}

\def\eb{\varepsilon}

\def\R {\mathbb{R}}

\def \pt {\partial_t}

\def\<{\left<}
\def\>{\right>}
\def\Ree{\operatorname{Re}}
\def\Imm{\operatorname{Im}}
\def\Nx{\px}
\def\Dx{\px^2}
\def\({\left(}
\def\){\right)}
\def\Cal{\mathcal}
\def\Bbb{\mathbb}

\newtheorem{proposition}{Proposition}[section]
\newtheorem{theorem}[proposition]{Theorem}
\newtheorem{corollary}[proposition]{Corollary}
\newtheorem{lemma}[proposition]{Lemma}

\theoremstyle{definition}
\newtheorem{definition}[proposition]{Definition}

\newtheorem{remark}[proposition]{Remark}

\numberwithin{equation}{section}

\def \no#1#2#3 {{\bf #1} (#3), #2.}
\def \eds#1#2#3 {#1, #2, #3.}

\title[Inertial manifolds for RDA equations] {Inertial manifolds for 1D reaction-diffusion-advection systems. Part II: periodic boundary conditions}
\author[A. Kostianko and S. Zelik]{ Anna Kostianko${}^1$ and Sergey Zelik${}^1$}
\address{${}^1$
University of Surrey, Department of Mathematics,
Guildford, GU2 7XH, United Kingdom, a.kostianko@surrey.ac.uk, s.zelik@surrey.ac.uk.}
\begin{document}
\begin{abstract} This is the second part of our study of the Inertial Manifolds for 1D systems of reaction-diffusion-advection equations initiated in \cite{KZI} and it is devoted to the case of periodic boundary conditions. It is shown that, in contrast to the case of Dirichlet or Neumann boundary conditions, considered in the first part, Inertial Manifolds may not exist in the case of systems endowed by periodic boundary conditions. However, as also shown,  inertial manifolds still exist in the case of scalar reaction-diffusion-advection equations. Thus, the existence or non-existence of inertial manifolds for this class of dissipative systems strongly depend on the choice of boundary conditions.
\end{abstract}
\thanks{This work is partially supported by  the grants  14-41-00044 and 14-21-00025 of RSF as well as  grants 14-01-00346  and 15-01-03587 of RFBR}
\maketitle
\tableofcontents
\def\R{\Bbb R}
\def\Dt{\partial_t}
\def\Dx{\partial^2_x}
\def\Nx{\partial_x}
\def\eb{\varepsilon}
\section{Introduction}
This paper can be considered as a continuation of our study of the Inertial Manifolds (IMs) for the so-called 1D reaction-diffusion-advection (RDA) systems of the form
\begin{equation}\label{0.main}
\partial_tu-\Dx u+u+f(u)\partial_x u+g(u)=0,\ \ x\in(-\pi,\pi),
\end{equation}
initiated in \cite{KZI} although can be read independently. Here $u=u(t,x)=(u_1,\cdots, u_m)$ is an unknown vector-valued function and $f$ and $g$ are given nonlinearities which are assumed belonging to the space $C_0^\infty$.
\par
It is well-known that the existence of IM usually requires strong extra assumptions on the dissipative system considered. For instance, for the abstract parabolic equation in a Hilbert space $H$:
\begin{equation}\label{0.abs}
\Dt u+Au=F(u),
\end{equation}
where $A: D(A)\to H$ is a linear positive self-adjoint operator with compact inverse and $F: H^\beta\to H$, $H^\beta:=D(A^{\beta/2})$, is a nonlinear globally Lipschitz operator, the spectral gap conditions read
\begin{equation}\label{0.sg}
\frac{\lambda_{N+1}-\lambda_N}{\lambda_{N+1}^{\beta/2}+\lambda_N^{\beta/2}}>L.
\end{equation}
Here $N$ is the dimension of the IM, $\{\lambda_n\}_{n=1}^\infty$ are the eigenvalues of $A$ enumerated in the non-decreasing order, $0\le \beta<2$ and $L$ is a Lipschitz constant of the nonlinearity $F$, see \cites{FST,mik,28,rom-man,Zel2} for more details. If this condition is satisfied, then there exists a Lipschitz ($C^{1+\eb}$-smooth for some small positive $\eb>0$ and normally hyperbolic if $F$ is smooth enough) invariant manifold of dimension $N$ in $H$ with the exponential tracking property. Thus, restricting equation \eqref{0.abs} to this manifold, we get a system of ODEs describing the limit dynamics generated by \eqref{0.abs} - the so-called inertial form (IF) of this equation. It is also known that, at least on the level of abstract equation \eqref{0.abs}, the spectral gap conditions \eqref{0.sg} are sharp and the IM may not exist if they are violated. Moreover, in this case the associated dynamics may also be infinite-dimensional despite the fact that the global attractor exists and has finite box-counting dimension, see \cites{EKZ,sell-counter,rom-3,Zel2} for the details.
\par
In the case of RDA equations \eqref{0.main}, $A:=-\Dx+1$ endowed by the proper boundary conditions, $H:=L^2(-\pi,\pi)$ and the nonlinearity $F(u):=f(u)\partial_xu+g(u)$ maps $H^1(-\pi,\pi)$ to $H$ (and can be made globally Lipschitz after the proper cut-off procedure), so $\beta=1$, $\lambda_n\sim n^2$ and the spectral gap conditions read
\begin{equation}
\frac{\lambda_{N+1}-\lambda_N}{\lambda_N^{1/2}+\lambda_{N+1}^{1/2}}\sim C\frac{(N+1)^2-N^2}{N+N+1}=C>L,
\end{equation}
where $C$ is independent of $N$. Thus, the nonlinearity in the RDA system \eqref{0.main} is in a sense critical from the point of view of the IM theory and the spectral gap conditions are satisfied only in the case where the Lipschitz constant $L$ of the nonlinearity is small enough (no matter what $N$ is). By this reason, the existence or non-existence of IMs for RDA equations with arbitrarily large nonlinearities was a long-standing open problem. We also remind that, in the scalar case $m=1$ there is the so-called Romanov theory which allows us to construct the IF with Lipschitz continuous nonlinearities without using IMs, see \cites{rom-th,rom-th1} (and also \cites{Kuk,Zel2}) but this result is essentially weaker than what we may get from IMs, namely, this IF works on the attractor only, does not possess any normal hyperbolicity/exponential tracking property and the smoothness of the associated vector field is restricted to $C^{0,1}$ (Lipschitz continuity only). Note that the regularity of the reduced IF equations is a crucial point here since the $C^{\alpha}$-smooth IFs with $\alpha<1$ can be constructed based  only on the Man\'e projection theorem and the fact that the box-counting dimension of the attractor is finite, see \cites{EKZ,28}, and this "reduction" works even in the examples where the dynamics on the attractor is clearly not finite-dimensional, see \cite{Zel2} for the details.
\par
 On the other hand,
as conjectured, such Lipschitz continuous IFs may be natural extensions of the concept of the IM to dissipative systems which do not satisfy the spectral gap conditions and do not possess IMs. However, to the best of our knowledge, up to the moment there was only one candidate where the Romanov theory works and the existence of the IM was unknown and this is exactly the class of scalar 1D RDA equations. Thus, one of the motivations of our study is to clarify at least on this model example whether   the existence  of Lipschitz continuous IFs is caused by the existence of IMs or the Romanov theory is indeed a step beyond the IMs. Another source of interest is related with the fact that Burgers or coupled Burgers equations (which are the particular cases of RDA equations) are often considered as simplified models for the Navier-Stokes equations and turbulence, so clarifying the situation with RDA equations may bring some light on the main open problem of the IM theory, namely, existence or non-existence of IMs for the Navier-Stokes equations.
\par
As follows from our investigation,  the existence of IMs for 1D RDA equations strongly depends on the type of boundary conditions chosen. We have considered three types of BC: Dirichlet, Neumann and periodic ones. As shown in the first part of our study \cite{KZI}, in the case of Dirichlet boundary conditions the problem can be settled by transforming our equation to the new one for which the spectral gap conditions will be satisfied using the trick with the non-local change of variables $u=a(t,x)w$, where $a$ is a properly chosen matrix depending on~$w$. Indeed, the new independent variable $w$ solves
\begin{multline}\label{0.trans}
\Dt w - \partial^2_x w = \{a^{-1}(2 \partial_xa - f(aw)a)\partial_x w\} +\\+ \{a^{-1}[\partial_x^2 a - \Dt a - f(aw)\partial_x a]w - a^{-1}g(aw)\}:=\mathcal F_1(w)+\mathcal F_2(w).
\end{multline}
This guesses the choice of the matrix $a$. The naive one would be to fix it as a solution of the following ODE
\begin{equation}\label{0.wrong}
\frac{d}{dx}a = \frac12 f(aw)a=\frac12f(u)a, \ \ a|_{x=-\pi} =\operatorname{Id}.
\end{equation}
Then the operator $\mathcal F_1(w)$ disappears, but the second one will still consume smoothness ($\mathcal F_2:H^1\to H$) due to the presence of $\Dx a$ and $\partial_ta$ and we will achieve nothing. However, a bit more clever choice
\begin{equation}\label{0.right}
\frac{d}{dx}a = \frac12 f(P_K(aw))a=\frac12f(P_Ku)a, \ \ a|_{x=-\pi} =\operatorname{Id},
\end{equation}
where $P_K$ is the orthoprojector to the first $K$ eigenvectors of $A:=-\Dx+1$, actually solves the problem. Indeed, in this case, $a$ depends on  $w$ (or $u$)  through the smoothifying operator $P_Kw$, so $\Dx a$ and $\partial_t a$ will not consume smoothness and the operator $\mathcal F_2$ will map $H^1(-\pi,\pi)$ to $H^1(-\pi,\pi)$. On the other hand, the map
$$
\mathcal F_1(w):=a^{-1}(f(P_K(aw))-f(aw))a\partial_x w
$$
can be made small (as an operator from $H^1(-\pi,\pi)$ to $H$) by fixing $K$ large enough, see \cite{KZI} for the details. Finally, as shown in \cite{KZI}, the map $u\to w$ is a diffeomorphism in the neighbourhood of the global attractor if $K$ is large enough and the transformed equation \eqref{0.trans} satisfies (after the proper cut-off procedure) the spectral gap conditions and possesses the IM. This gives the positive answer on the IM's existence problem for 1D RDA systems endowed by Dirichlet boundary conditions.
\par
The case of Neumann boundary conditions is more delicate due to the fact that the transform $u=aw$ {\it does not preserve} these boundary conditions and as a result, we would have the nonlinear and non-local boundary conditions for the transformed equation \eqref{0.trans}. Since nothing is known about the IMs for such type of BC even in the simplest cases, making this transform does not look as a good idea. Fortunately, there is an alternative way to handle this problem, namely, to reduce the Neumann BC to the Dirichlet one by differentiating the equations in $x$. Indeed, let $v=\partial_x u$. Then functions $(u,v)$ solve
\begin{equation}\label{0.ND}
\begin{cases}
\Dt u-\Dx u+u+f(u)v+g(u)=0,\ \ \partial_xu\big|_{x=-\pi}=\partial_x u\big|_{x=\pi}=0\\
\Dt v-\Dx v+v+f(u)\partial_x v+f'(u)v^2+g'(u)v=0,\ \ v\big|_{x=-\pi}=v\big|_{x=\pi}=0.
\end{cases}
\end{equation}
Since the first equation does not contain first derivatives in $x$, it is enough to transform the second component $v=a(t,x)w$ and this component has Dirichlet boundary conditions and the above mentioned problem with boundary conditions is overcome, see \cite{KZI} for the details. Thus, in the case of Neumann boundary conditions, the initial RDA system can be embedded into a larger RDA system which possesses an IM, so the answer on the question about the existence of IMs in this case is also positive.
\par
This paper is devoted to the most complicated case of periodic boundary conditions. As we will see, there is a principal difference between the scalar ($m=1$) and vector ($m>1$) cases. In the scalar case,
it is possible to modify the equation for $a$ as follows
\begin{equation}\label{0.right-per}
\frac{d}{dx}a = \frac12 [f(P_K(aw))-\<f(P_K(aw)\>]a=\frac12[f(P_Ku)-\<f(P_Ku)\>]a, \ \ a|_{x=-\pi} =\operatorname{Id},
\end{equation}
where $\<U\>:=\frac1{2\pi}\int_{-\pi}^{\pi}U(x)\,dx$ is the spatial mean of the function $U$. This extra term makes the function $a$ $2\pi$-periodic in space (so the associated transform will preserve the periodic boundary conditions). On the other hand, it leads to the extra term
$$
\mathcal F_3(w):=\<f(P_K(aw)\>\partial_x w
$$
in the right-hand side of the transformed equations \eqref{0.trans} which is {\it not small} and still consumes smoothness ($\mathcal F_3$ maps $H^1(-\pi,\pi)$ to $H$ only). Nevertheless, as shown below, this term does not destroy the construction of the IM since it has very special structure. Thus, in the scalar case and periodic boundary conditions, the answer on the question about the existence of IMs is also positive, namely, the following theorem can be considered as one of two main results of this paper.
\begin{theorem}\label{Th0.main1} Let  the nonlinearities $f,g\in C_0^\infty(\R)$. Then
the RDA equation
\begin{equation}\label{0.per}
\Dt u-\Dx u+u+f(u)\partial_x u+g(u)=0, \ \ u\big|_{x=-\pi}=u\big|_{x=\pi},\ \ \partial_x u\big|_{x=-\pi}=\partial_x u\big|_{x=\pi}
\end{equation}
possesses an IM (after the proper cut-off procedure).
\end{theorem}
Note that the proof of this theorem differs essentially from the one given in the first part of our study for the case of Dirichlet boundary conditions. In particular, we have to use a special cut-off procedure similar to the one developed in \cite{mal-par} (see also \cites{kwean,KZ1,Zel2}) for the so-called spatial averaging method as well as the graph transform and invariant cones instead of the Perron method. The extra term $u$ is added only in order to have dissipativity and the global attractor in the periodic case as well and is not essential for IMs.
\par
We now turn to the vector case $m>1$ with periodic boundary conditions. In contrast to all previously mentioned cases, the answer on the question about the existence of IMs here is {\it negative}. Namely, the following theorem can be considered as the second main result of the present paper.
\begin{theorem}\label{Th0.main2} There exist $m>1$ and the nonlinearities $f\in C_0^\infty(\R^m,\mathcal L(\R^m,\R^m))$ and $g\in C_0^\infty(\R^m,\R^m)$ such that the associated RDA system \eqref{0.main} with periodic boundary conditions does not possess any finite-dimensional IM containing the global attractor. Moreover, the associated limit dynamics on the attractor is infinite-dimensional and, in particular, contains limit cycles with supra exponential rate of attraction.
\end{theorem}
As in the case of abstract parabolic equations, see \cites{EKZ,Zel2}, the proof of this result is based on the proper counterexample to the Floquet theory for linear equations with time-periodic coefficients. Such counterexamples are well-known and can be relatively easily constructed in the class of {\it abstract} parabolic equations, see \cite{kuch} for the details. However, to the best of our knowledge, finding such counterexamples in the class of parabolic PDEs and local differential operators was also a long standing open problem. In the present paper, we give a solution of this problem. Namely, we have found smooth space-time periodic functions $f(t,x)\in \mathcal L(\R^m,\R^m)$ and $g(t,x)\in\R^m$ such that the period map $U$ associated with the linear RDA system
\begin{equation}\label{0.floq}
\Dt u-\partial_x^2u+u+f(t,x)\partial_x u+g(t,x)u=0
\end{equation}
is a Volterra type operator such that its spectrum coincides with $\{0\}$. As a result, all solutions of problem \eqref{0.floq} decay {\it faster} than exponentially as $t\to\infty$ (actually, the decay rate is like $e^{-\kappa t^3}$ for some positive $\kappa$).
\par
Note also that, as shown in \cite{rom3}, the IM may not exist even in the scalar case of RDA equation and periodic boundary conditions if we allow the nonlinearities to contain the non-local terms like periodic Hilbert operators.
\par
The paper is organized as follows.
\par
We first study the scalar case.
Section \ref{s1} is devoted to  the properties of solutions of \eqref{0.right-per} and the associated diffeomorphisms of the phase space. In Section \ref{s2} we deduce the transformed equations and verify the basic properties of the transformed nonlinearities which are crucial for the inertial manifold theory, and the construction of the IM for the scalar case is given in Section \ref{s3} based on a special cut-off procedure in the spirit of \cite{mal-par} and invariant cones.
\par
The example of a system of eight RDA equations with periodic boundary conditions which does not possess any finite-dimensional inertial manifold is given in Section \ref{s4}. Finally, some generalizations of the obtained results are considered in Section \ref{s5}.

\section{Scalar case: an auxiliary diffeomorphism}\label{s1}

 In this section, we study the nonlinear transformation $u(t,x)=a(t,x)w(t,x)$ mentioned in the introduction. Namely, we define the map $W: H_{per}^1(-\pi,\pi)\to H_{per}^1(-\pi,\pi)$ via the expression $W(u)(x):=[a(u)(x)]^{-1}u(x)$ where the function $a(x)=a(u)(x)$ solves the equation
\begin{equation}\label{1.dir}
\frac{d}{dx}a =\frac12[f(P_Ku)-\<f(P_Ku)\>]a, \ \ a|_{x=-\pi} =1,
\end{equation}
where $K$ is large enough and $P_K$ is the orthoprojector to the first $2K+1$ Fourier modes. We recall that in our scalar  space periodic case, the eigenvalues of the Laplacian $A=-\Dx+1$ are $\lambda_0=1$, $\lambda_{2n-1}=\lambda_{2n}:= n^2+1$ for $n>0$ (for $n>0$ these eigenvalues have multiplicity two with the corresponding eigenfunctions $\sin(nx)$ and $\cos(nx)$). Thus, the orthonormal basis of eigenfunctions coincides with the  basis for the classical Fourier series and, in particular, the orthoprojector $P_K$ has the form:
\begin{multline}
(P_K u)(x)=\frac{a_0}{2}+\sum_{n=1}^Ka_n \cos(nx)+b_n\sin(nx),\\ a_n:=\frac1\pi\int_{-\pi}^\pi u(x)\cos(nx)\,dx,\ b_n:=\frac1\pi\int_{-\pi}^\pi u(x)\sin(nx)\,dx.
\end{multline}
The basic properties of the maps $a(u)$ and $W(u)$ are collected in the following lemmas
\begin{lemma}\label{bau}
For any $u\in H^1_{per}(-\pi,\pi)$ and any $K$ there exists a unique solution $a=a(u)\in C^\infty(\R)$ for problem \eqref{1.dir}. This solution is space periodic with the period $2\pi$ and the following estimate holds:
\begin{equation}
\|a\|_{W^{1,\infty}} + \|a^{-1}\|_{W^{1,\infty}}\le C,
\end{equation}
where constant $C$ is independent of $K$ and $u$. Moreover, the maps $u\to a(u)$ and $u\to a^{-1}(u)$ are $C^\infty$-differentiable as maps from $H^1_{per}(-\pi,\pi)$ to $W^{1,\infty}(-\pi,\pi)$ and the norms of their Frechet derivatives are bounded by constants which are independent of $u$ and $K$. In particular, the following global Lipschitz continuity holds:
\begin{equation}
\|a(u_1) -a(u_2)\|_{W^{1,\infty}} + \|a^{-1}(u_1) - a^{-1}(u_2)\|_{W^{1,\infty}} \le C\|u_1 - u_2\|_{H^1},
\end{equation}
where the constant $C$ is independent of $K$ and $u$.
\end{lemma}
\begin{proof}
Since equation \eqref{1.dir} is linear we can explicitly solve it and obtain
\begin{equation}\label{k}
a(u)(x) = e^{\frac12\int_{-\pi}^x f((P_Ku)(s)) - \<f(P_K u)\>\,ds}.
\end{equation}
All assertions of the lemma follow then from this explicit formula and the fact that $f\in C_0^\infty(\R)$. Indeed, the Frechet derivative $a'(u)\theta$, $\theta\in H^1_{per}(-\pi,\pi)$ satisfies
$$
a'(u)\theta=\frac12 a(u)\int_{-\pi}^x f'(P_Ku)P_K\theta -
\<f'(P_K u)P_K\theta\>\,ds
$$
and we see that
$$
\|a'(u)\theta\|_{W^{1,\infty}}\le C\|a(u)\|_{W^{1,\infty}}\|P_K\theta\|_{L^\infty}\le C\|\theta\|_{H^1}.
$$
This gives the desired uniform Lipschitz continuity. The higher Frechet derivatives may be estimated
analogously and the lemma is proved.
\end{proof}
\begin{corollary}\label{Cor1.direct} The map $u\to W(u)$ is $C^\infty$-smooth as a map from $H^1_{per}(-\pi,\pi)$ to itself and the norms of its Frechet derivatives are uniformly bounded with respect to $K$ (but depend on $\|u\|_{H^1}$). Moreover, the following estimate holds:
\begin{equation}\label{1.bound}
C^{-1}\|u\|_{H^1}\le \|W(u)\|_{H^1}\le C\|u\|_{H^1},
\end{equation}
where the constant $C>1$ is independent of $K$ and $u$.
\end{corollary}
Indeed, these assertions are immediate corollaries of Lemma \ref{bau}.
\begin{remark} Note that $a(u)$ is a smoothifying operator in the sense that $a(u)\in C^\infty(-\pi,\pi)$ if $u\in H^1(-\pi,\pi)$. However, the smoothifying norms of this operator will depend on $K$. Actually, the $H^2$-norm will be still uniform with respect to $K$ since
$$
\frac{d^2}{dx^2}a=\frac14(f(P_Ku)-\<f(P_Ku)\>)^2a+\frac12 af'(P_Ku)P_K\frac d{dx}u
$$
and $u\in H^1_{per}(-\pi,\pi)$, but the third derivative will contain the term $\frac{d^2}{dx^2}P_Ku$ which is not uniform with respect to $K$. Moreover, as we see from this formula, the $H^2$-norm of $a$ is not globally bounded with respect to $u$ (due to the presence of the linearly growing term $\frac{d}{dx}P_Ku$). That is the reason why we use $W^{1,\infty}$-norm in Lemma \ref{bau}.
\end{remark}
Recall that we want to verify that the map $W$ is a diffeomorphism, so we need to study the inverse map $U:w\to u$. To this end, we need to find the function $a$ if the function $w$ is known. Obviously, this function (if exists) should satisfy the equation
\begin{equation}\label{a}
\frac{d}{dx} a = \frac{1}{2} (f(P_K(aw)) - \<f(P_K(aw))\>)a, \ \ a|_{x=-\pi}=1.
\end{equation}
The rest of this section is devoted to the study of equation \eqref{a}. We start with the solvability problem.
\begin{lemma}\label{Lem1.sol} For any $w\in H^1_{per}(-\pi,\pi)$ and any $K$ (including $K=\infty$ which corresponds to $P_K=\operatorname{Id}$), there exists at least one solution $a\in C^\infty_{per}(-\pi,\pi)$ of equation \eqref{a}. Moreover,
\begin{equation}\label{1.abound}
\|a\|_{W^{1,\infty}}+\|a^{-1}\|_{W^{1,\infty}}\le C,
\end{equation}
where the constant $C$ is independent of $K$ and $w$.
\end{lemma}
\begin{proof} It is convenient to make the change of variable $a=e^y$ and replace equation \eqref{a} by the following one
\begin{equation}\label{1.y}
\frac{d}{dx}y=\frac12(f(P_K(e^yw))-\<f(P_K(e^yw))\>),\ \ y\big|_{x=-\pi}=0
\end{equation}
or, in the equivalent integral form,
\begin{equation}\label{1.y-int}
y(x)=\frac12\int_{-\pi}^x(f(P_K(e^{y(s)}w(s)))-\<f(P_K(e^yw))\>)\,ds:=I(y)(x).
\end{equation}
Note that the condition $y(\pi)=0$ is automatically satisfied, so the function $a$ is automatically $2\pi$ periodic in $x$. Moreover, as not difficult to see, the operator $I$ is continuous and compact as an operator from $C[-\pi,\pi]$ to itself and is globally bounded
$$
\|I(y)\|_{C[-\pi,\pi]}\le C,
$$
where $C$ is independent of $y$ and $K$. Thus, $I$ maps the closed $R$-ball in the space $C[-\pi,\pi]$ to itself if $R\ge C$ and, thanks to the Schauder fixed point theorem, equation \eqref{1.y-int} possesses at least one solution $y\in C[-\pi,\pi]$ belonging to this ball. All other properties stated in the lemma are immediate corollaries of equation \eqref{1.y-int} and the lemma is proved.
\end{proof}
Our next task is to verify the uniqueness of the solution $a$ and its smooth dependence on the function $w$. To this end, we start with the following linear problem which corresponds to the linearization of \eqref{1.y} with $K=\infty$:
\begin{equation}\label{1.lin}
\frac{d}{dx}\xi(x)=\varphi(x)\xi(x)-\<\varphi\xi\>+h(x),\ \ \xi\big|_{x=-\pi}=0.
\end{equation}

\begin{lemma}\label{Lem1.lin} Let $\varphi,h\in L^1(-\pi,\pi)$ and $\<h\>=0$. Then, problem \eqref{1.lin} possesses a unique solution $\xi\in W^{1,1}_{per}(-\pi,\pi)$ and this solution is given by the following expression:
\begin{equation}\label{solynon}
\xi(x)= \int_{-\pi}^x e^{\int_s^x \varphi(\chi)d\chi}(-D + h(s))ds,
\end{equation}
where
\begin{equation}\label{const}
D= \<\varphi\xi\>=  \frac{\int_{-\pi}^\pi h(s)e^{\int_s^\pi\varphi(\chi)d\chi}ds}{\int_{-\pi}^\pi e^{\int_s^\pi \varphi(\chi)d\chi}ds}.
\end{equation}
Moreover, the following estimate holds:
\begin{equation}\label{ae}
\|\xi\|_{W^{1,1}} \le C \|h\|_{L^1},
\end{equation}
where $C=C(\|\varphi\|_{L^1})$ is independent of $h$,
and, therefore, the linear solution operator
$$
\Upsilon=\Upsilon_\varphi : L^1(-\pi,\pi) \to W^{1,1}(-\pi,\pi),\ \  \Upsilon_\varphi h:=\xi
$$
is well-defined.
\end{lemma}
\begin{proof} Indeed, denoting $D:=\<\varphi\xi\>$ and solving the linear ODE with the initial data $\xi\big|_{x=-\pi}=0$, we get \eqref{solynon}. The explicit value of $D$ is then computed assuming that $\xi(\pi)=0$ and inserting this value to the left-hand side of \eqref{solynon}. Estimate \eqref{ae} is then an immediate corollary of \eqref{solynon} and the lemma is proved.
\end{proof}
\begin{corollary}\label{Cor1.inf} Let, in addition, $\varphi, h\in L^s(-\pi,\pi)$ for some $1\le s\le\infty$, then
\begin{equation}\label{1.ups-inf}
\|\Upsilon_\varphi\|_{\mathcal L(L^s,W^{1,s})}\le C,
\end{equation}
where the constant $C$ depends on the $L^s$-norm of $\varphi$.
\end{corollary}
Indeed, this estimate is also an immediate corollary of \eqref{solynon}.
\par
We now ready to study the case $K<\infty$. Namely, let us consider the following equation:
\begin{equation}\label{aK}
\frac d{dx} \xi=\varphi(x)(P_K(\psi\xi))(x)-\<\varphi P_K(\psi\xi)\>+h(x),\ \ \xi\big|_{x=-\pi}=0.
\end{equation}
\begin{lemma}\label{Lem1.linK} Let $\varphi,h\in L^2(-\pi,\pi)$ with $\<h\>=0$ and $\psi\in H^{1}_{per}(-\pi,\pi)$. Then, there exists $K_0=K_0(\|\varphi\|_{L^2}+\|\psi\|_{H^{1}})$ such that, for all $K>K_0$, problem \eqref{aK} is uniquely solvable in the space $H^{1}_{per}(-\pi,\pi)$. Therefore, the linear solution operator
\begin{equation}\label{1.upsK}
\Upsilon^K=\Upsilon^K_{\varphi,\psi}: L^2(-\pi,\pi)\to H^{1}(-\pi,\pi),\ \ \Upsilon_{\varphi,\psi}^Kh:=\xi
\end{equation}
is well defined and possesses the estimate
\begin{equation}\label{1.ups-est}
\|\Upsilon_{\varphi,\psi}^K\|_{\mathcal L(L^2,H^{1})}\le C,
\end{equation}
where the constant $C$ depends on $\|\varphi\|_{L^2}$ and $\|\psi\|_{H^{1}}$, but is independent of $K>K_0$. Moreover, if in addition the assumptions of Corollary \ref{Cor1.inf} are satisfied for some $s>2$ then the analogue of \eqref{1.ups-inf} also holds uniformly with respect to $K>K_0$.
\end{lemma}
\begin{proof} We rewrite problem \eqref{aK} using operator $\Upsilon$ and the fact that $\psi\xi=P_K(\psi\xi)+(1-P_K)(\psi\xi)$ in the form
\begin{equation}\label{1.good}
\xi=\Upsilon_{\varphi\psi}\(-\varphi(1-P_K)(\psi\xi)+\<\varphi(1-P_K)(\psi\xi)\>\)+\Upsilon_{\varphi\psi} h:=\mathcal R_{K,\varphi,\psi}\xi+\Upsilon_{\varphi\psi} h.
\end{equation}
Then, it is sufficient to verify that the operator $\mathcal R_{K,\varphi,\psi}$ is a contraction in $H^{1}_{per}(-\pi,\pi)$. Indeed, due to the interpolation inequality and the fact that $H^{1}$ is an algebra, we have
\begin{multline}\label{1.small}
\|(1-P_K)(\psi\xi)\|_{L^\infty}\le \\\le C\|(1-P_K)(\psi\xi)\|_{L^2}^{1/2}\|(1-P_K)(\psi\xi)\|_{H^{1}}^{1/2}\le CK^{-1/2}\|\psi\xi\|_{H^{1}}\le CK^{-1/2}\|\xi\|_{H^{1}},
\end{multline}
where the constant $C$ depends on $\|\psi\|_{H^{1}}$, but is independent of $K$.
  Then, due to Corollary~\ref{Cor1.inf},
\begin{multline}\label{1.ei}
\|\mathcal R_{K,\varphi,\psi}\xi\|_{H^{1}}\le C\|\varphi(1-P_K)(\psi\xi)\|_{L^2}\le\\\le C\|\varphi\|_{L^2}\|(1-P_K)(\psi\xi)\|_{L^\infty}\le Q(\|\varphi\|_{L^2}+\|\psi\|_{H^{1}})K^{-1/2}\|\xi\|_{H^{1}},
\end{multline}
where the function $Q$  is independent of $K$. Fixing now
$$
K_0:=4Q(\|\varphi\|_{L^2}+\|\psi\|_{H^{1}})^2,
$$
we see that for all $K>K_0$ the operator $\mathcal R_{K,\varphi,\psi}$ is a contraction with the contraction factor which is less than $1/2$. Therefore, equation \eqref{1.good} is uniquely solvable in $H^{1}_{per}$ and the following estimate holds:
\begin{equation}\label{1.better}
\|\xi\|_{H^{1}}\le 2\|\Upsilon_{\varphi\psi}h\|_{H^{1}}\le C\|h\|_{L^2},
\end{equation}
where the constant $C$ depends on $\|\varphi\|_{L^2}$ and $\|\psi\|_{H^{1}}$, but is independent of $K>K_0$. The remaining statements of the lemma are immediate corollaries of this estimate and the lemma is proved.
\end{proof}
We are finally ready to establish the analogue of Lemma \ref{bau} for the map $w\to a(w)$ defined as a solution of equation \eqref{a}.
\begin{lemma}\label{Lem1.baw} For any $R>0$ there exists $K_0=K_0(R)$ such that for any $w\in H^{1}_{per}(-\pi,\pi)$, $\|w\|_{H^{1}}\le R$ and every $K>K_0$, equation \eqref{a} possesses a unique solution $a=a(w)\in C^\infty_{per}(-\pi,\pi)$. Moreover, the map $w\to a(w)$ is $C^\infty$-differentiable as the map from
$$
B(R,0,H^{1}):=\left\{w\in H^{1}_{per}(-\pi,\pi),\ \|w\|_{H^{1}}<R\right\}
$$
to $W^{1,\infty}_{per}(-\pi,\pi)$ and the norms of its Frechet derivatives depend on $R$, but are independent of the value of the parameter $K>K_0$.
\end{lemma}
\begin{proof} The existence of the solution $a$ is verified in Lemma \ref{Lem1.sol}. Let us verify the uniqueness. Instead of working with equation \eqref{a}, we will work with the equivalent equation \eqref{1.y}. Indeed, let $y_1$ and $y_2$ be two solutions of this equation which correspond to the same $w\in B(R,0,H^{1})$ and let $\bar y:=y_1-y_2$. Then this function solves
\begin{equation}\label{1.dif}
\frac d{dx}\bar y=\varphi_{y_1,y_2}(x)P_K(\psi_{y_1,y_2}\bar y)-\<\varphi_{y_1,y_2}P_K(\psi_{y_1,y_2}\bar y)\>,
\end{equation}
where
$$
\varphi_{y_1,y_2}:=\frac12\int_0^1f'(P_K(se^{y_1}w+(1-s)e^{y_2}w))\,ds,\ \ \psi_{y_1,y_2}:=w\int_0^1e^{sy_1+(1-s)y_2}\,ds.
$$
Since $f\in C^\infty_0(\R)$ and $y_i$ are uniformly bounded in $W^{1,\infty}$, we have
\begin{equation}\label{1.ybound}
\|\varphi_{y_1,y_2}\|_{L^2}\le C,\ \ \|\psi_{y_1,y_2}\|_{H^{1}}\le C\|w\|_{H^{1}}\le CR
\end{equation}
where the constant $C$ is independent of $K$. Thus, according to Lemma \ref{Lem1.linK}, $\bar y=0$ is a unique solution of \eqref{1.dif} if $K>K_0(R)$ and the uniqueness is proved. Let us now estimate the norm of the Frechet derivative of the map $w\to a(w)$ (the differentiability can be verified in a standard way and we left its proof to the reader). Let $w\in B(R,0,H^{1})$, $\theta\in H^{1}_{per}(-\pi,\pi)$ and $\xi:=y'(w)\theta$. Then, this function solves
\begin{multline}
\frac d{dx}\xi=\frac12 \(f'(P_K(e^yw))P_K(e^yw\xi)-\<f'(P_K(e^yw))P_K(e^yw\xi)\>\)+\\+
\frac12\(f'(P_K(e^yw))P_K(e^y\theta)-\<f'(P_K(e^yw))P_K(e^y\theta)\>\),\ \xi\big|_{x=-\pi}=0.
\end{multline}
This equation has the form of equation \eqref{aK} with
$$
\varphi:=\frac12f'(P_K(e^yw),\ \psi:=e^yw,\ h=\frac12\(f'(P_K(e^yw))P_K(e^y\theta)-\<f'(P_K(e^yw))P_K(e^y\theta)\>\).
$$
Moreover, the functions $\varphi$ and $\psi$ satisfy exactly the same bounds as in \eqref{1.ybound} and, consequently, according to Lemma \ref{Lem1.linK},
$$
\|\xi\|_{W^{1,\infty}}=\|\Upsilon_{\varphi,\psi}^Kh\|_{W^{1,\infty}}\le C\|h\|_{L^\infty}
$$
if $K>K_0(R)$. It remains to note that
$$
\|h\|_{L^\infty}\le C\|P_K(e^y\theta)\|_{L^\infty}\le C\|e^y\theta\|_{H^{1}}\le C\|\theta\|_{H^{1}},
$$
where $C$ is independent of $K$. This gives the following estimate
\begin{equation}
\|y'(w)\|_{\mathcal L(H^{1},W^{1,\infty})}\le C,
\end{equation}
where $C$ is independent of $K$ and the desired uniform bound for the first Frechet derivative is obtained. Higher derivatives can be estimated analogously and the lemma is proved.
\end{proof}
We combine the obtained results in the following theorem.
\begin{theorem}\label{Th1.main} For any $R>0$ there exists $K_0=K_0(R)$ such that
the map $W: H^1(-\pi,\pi)_{per} \to H^1_{per}(-\pi,\pi)$ is $C^{\infty}-$ diffeomorphism between $B(R,0,H^1)$ and $W(B(R,0,H^1)) \subset H^1$ if $K> K_0(R)$. Moreover, the norms of  $W$, $U:=W^{-1}$  and their derivatives are independent of $K$ and the following embeddings hold:
\begin{equation}\label{1.ul}
B(C^{-1}R,0,H^1_{per})\subset W(B(R,0,H^1_{per}))\subset B(CR,0,H^1_{per})
\end{equation}
for some constant $C>1$ which is independent of $K$ and $R$.
\end{theorem}
Indeed, embeddings follow from inequalities \eqref{1.bound} and the remaining properties are actually proved in Corollary \ref{Cor1.direct} and Lemma \ref{Lem1.baw}.

\section{Scalar case: the transformed equation}\label{s2}

The aim of this section is to make the change $w=W(u)$ of the independent variable $u$ and study the properties of the nonlinearities involved in the transformed equation. Recall that the transform $W(u)$ is a diffeomorphism on  a large ball $B(R,0,H^1_{per})$ only (where $R$ depends on the parameter $K$), so we need to do this transform not in the whole phase space $\Phi:=H^1_{per}(-\pi,\pi)$, but only on the absorbing ball of the corresponding solution semigroup. By this reason, we start our exposition with a theorem which guarantees the well-posedness and dissipativity of the solution semigroup (although in this section we need this result for the scalar equation only, we state below the theorem for the vector case as well). Namely,
let us consider the following RDA system with periodic boundary conditions:
\begin{equation}\label{eq_1D}
\begin{cases}
 \Dt u -\partial^2_x u+u  + f(u)\partial_x u+g(u)= 0,\ x\in(-\pi, \pi),\\
 u|_{t=0}=u_0\in \Phi,
\end{cases}
\end{equation}
where $u(x,t)=(u_1,\cdots,u_m)$ is an unknown vector function, $f$ and $g$ are given nonlinear smooth functions with finite support.
\begin{theorem}\label{dis}
Let the above assumptions hold. Then for any $u_0 \in H^1_{per}(-\pi,\pi)$ there exists a unique solution of equation \eqref{eq_1D}
\begin{equation}
u \in C([0,T], H^1_{per}(-\pi,\pi))\cap L^2([0,T], H^2(-\pi,\pi)),\ \  T>0,
\end{equation}
satisfying $u|_{t=0} = u_0$ and, therefore, the solution semigroup $S(t)$ is well-defined in the phase  space $\Phi$ via
\begin{equation}
S(t): \Phi \to \Phi, \ \ S(t)u_0:=u(t).
\end{equation}
Moreover the following estimates hold for any solution $u(t)$ of problem \eqref{eq_1D}
\par
1. Dissipativity:
\begin{equation}
\|u(t)\|_{\Phi} \le  Ce^{-\gamma t}\|u_0\|_\Phi +C,
\end{equation}
where $\gamma, C$ are some positive constants;
\par
2. Smoothing property:
\begin{equation}\label{2.smo}
\|u(t)\|_{H^2}\le t^{-1/2}Q(\|u(0)\|_{\Phi}) + C_*,
\end{equation}
where the monotone function $Q$ and positive constant $C_*$ are independent of $t>0$.
\end{theorem}
\begin{proof} We give below only the schematic derivation of the stated estimates lefting the standard details to the reader.
\par
{\it Step 1. $L^2$-estimate.} Multiplying equation \eqref{eq_1D} by $u$, integrating over $x$ and using that both $f$ and $g$ have finite support, we get
\begin{equation}
\frac12\frac d{dt}\|u\|^2_{L^2}+\|\Nx u\|^2_{L^2}+\|u\|^2_{L^2}\le C_1\|\Nx u\|_{L^2}+C_2
\end{equation}
and after applying the Gronwall inequality, we arrive at
\begin{equation}\label{2.l2dis}
\|u(t)\|^2_{L^2}+\int_t^{t+1}\|\Nx u(s)\|^2_{L^2}\,ds\le C\|u(0)\|^2_{L^2}e^{-\delta t}+C_*
\end{equation}
for some positive $C_*$, $\delta$ and $C$ which are independent of $t$ and $u$.
\par
{\it Step 2. $H^1$-estimate.} Multiplying equation \eqref{eq_1D} by $-\Dx u$, integrating by parts and using again the fact that $f$ and $g$ have finite supports, we get
\begin{multline}
\frac12\frac d{dt}\|\Nx u\|^2_{L^2}+\|\Dx u\|^2_{L^2}+\|\Nx u\|^2_{L^2}\le\\\le C\|\Dx u\|_{L^2}+C_1\|\Nx u\|_{L^2}\|\Dx u\|_{L^2}\le \frac12\|\Dx u\|^2_{L^2}+C(\|\Nx u\|^2_{L^2}+1)
\end{multline}
Applying the Gronwall inequality to this relation and using \eqref{2.l2dis} for estimating the right-hand side, we arrive at
\begin{equation}\label{2.h1dis}
\|u(t)\|^2_{H^1}+\int_t^{t+1}\|u(s)\|^2_{H^2}\,ds\le C\|u(0)\|^2_{H^1}e^{-\delta t}+C_*
\end{equation}
which gives the desired dissipative estimate in $H^1$.
\par
{\it Step 3. Smoothing property.} Multiplying equation \eqref{eq_1D} by $\partial_x^4 u$, integrating by parts, using again that $f$ and $g$ have finite supports and the interpolation inequality $\|v\|_{L^\infty}^2\le C\|v\|_{L^2}\|\Nx v\|_{L^2}$, we get
\begin{multline}\label{2.gronh2}
\frac12\frac d{dt}\|\partial^2_x u\|^2_{L^2}+\|\partial_x^3u\|^2_{L^2}+\|\partial^2_x u\|^2_{L^2}\le\\\le C\|\partial^3_xu\|_{L^2}(\|\partial^2_x u\|_{L^2}+\|\partial_x u\|^2_{L^\infty}+\|\partial_x u\|_{L^2})\le C\|\partial^3_xu\|_{L^2}(\|\Nx u\|_{L^2}+1)(\|\Dx u\|_{L^2}+1)\le\\\le \frac12\|\partial^3_x u\|^2_{L^2}+C(\|\Nx u\|^2_{L^2}+1)(\|\Dx u\|^2_{L^2}+1).
\end{multline}
Applying the Gronwall inequality to this relation and using  \eqref{2.h1dis}, we arrive at the dissipative estimate in $H^2$:
\begin{equation}\label{2.h2dis}
\|u(t)\|_{H^2}^2+\int_t^{t+1}\|u(s)\|^2_{H^3}\,ds\le Q(\|u(0)\|_{H^2})e^{-\gamma t}+C_*.
\end{equation}
for some monotone increasing function $Q$ and positive constants $\gamma$ and $C_*$ which are independent of $t$. Finally, to obtain the smoothing property, we assume that $t\le1$, multiply inequality \eqref{2.gronh2} by $t$ and apply the Gronwall inequality with respect to the function $Y(t):=t\|\Dx u(t)\|^2_{L^2}$. This gives estimate \eqref{2.smo} for $t\le1$. The estimate for $t\ge1$ can be obtained combining estimate \eqref{2.h2dis} with \eqref{2.smo} for $t\le1$.
\par
{\it Step 4. Uniqueness and Lipschitz continuity.} Let $u_1$ and $u_2$ be two solutions of equation \eqref{eq_1D} and let $\bar u=u_1-u_2$. Then, this function solves
\begin{equation}\label{2.dif}
\Dt\bar u-\Dx \bar u+\bar u+[f(u_1)\Nx u_1-f(u_2)\Nx u_2]+[g(u_1)-g(u_2)]=0.
\end{equation}
Using the fact that $H^1$ is an algebra together with estimate \eqref{2.h1dis}, we get
$$
\|f(u_1)\Nx u_1-f(u_2)\Nx u_2\|_{L^2}\le C\|u_1-u_2\|_{H^1},\ \ \|g(u_1)-g(u_2)\|_{L^2}\le C\|u_1-u_2\|_{L^2},
$$
where the constant $C$ depends on the $H^1$-norms of $u_1(0)$ and $u_2(0)$. Multiplying now equation \eqref{2.dif} by $-\Dx \bar u+\bar u$ and using these estimates, we end up after the standard transformations with the following inequality:
$$
\frac d{dt}\|\bar u\|^2_{H^1}+\|\bar u\|^2_{H^2}\le \tilde C\|\bar u\|^2_{H^1}
$$
and, therefore,
\begin{equation}\label{2.ulip}
\|u_1(t)-u_2(t)\|_{H^1}^2\le e^{\tilde C t}\|u_1(0)-u_2(0)\|_{H^1}^2,
\end{equation}
where the constant $\tilde C$ depends on the $H^1$-norms of $u_1$ and $u_2$, but is independent of $t$. Thus, the uniqueness is verified. The existence of a solution can be proved using e.g., the Galerkin approximations, see \cites{bv1,tem} and the theorem is proved.
\end{proof}
The proved theorem guarantees the existence of a global attractor for the solution semigroup $S(t)$. For the convenience of the reader, we recall the definition of the global attractor and state the corresponding result, see \cites{bv1,tem} for more details.

\begin{definition}\label{Def2.attr}
A set $\mathcal{A}$ to be called a global attractor for the solution semigroup $S(t)$ generated by equation \eqref{eq_1D} if it satisfies the following properties:
\par
1. The set $\mathcal{A}$ is  compact  in $\Phi:=H^1_{per}(-\pi,\pi)$;
\par
2. The set $\mathcal{A}$ is invariant with respect to the  semigroup $S(t)$, i.e., $S(t)\mathcal{A} = \mathcal{A}$, $t\ge0$;
\par
3. The set $\mathcal{A}$ is  attracting, i.e., for any bounded set $B\subset \Phi$ and any neighbourhood $\mathcal{O}$ of the attractor $\mathcal{A}$ there exists time $T = T(B,\mathcal{O})$ such that
\begin{equation}
S(t)B \subset \mathcal{O}(\mathcal{A}),\ \ \text{ for all } t\ge T.
\end{equation}
\end{definition}

The next result is a standard corollary of the proved Theorem \ref{dis}
\begin{theorem}\label{Th2.attr}
 Under the assumptions of Theorem \ref{dis} the solution semigroup $S(t)$ of \eqref{eq_1D} possesses a global attractor $\mathcal A$ in the phase space $\Phi=H^1_{per}(-\pi,\pi)$. Moreover this attractor is a bounded set in $H^2_{per}(-\pi,\pi)$.
\end{theorem}
\begin{remark}\label{Rem2.attr} We recall that, as a rule, the nonlinearities $f$ and $g$ do not have finite support, but satisfy some {\it dissipativity} and {\it growth} restrictions which allow to establish the dissipativity of the corresponding solution semigroup and the existence of a global attractor, say, in the phase space $\Phi$, see e.g., \cite{BPZ} for the case of coupled Burgers equations. After that, since we are interested in the long time behavior of solutions only, we cut off the nonlinearities outside of some neighbourhood of the attractor making them $C_0^\infty$ on the one hand and without changing the global attractor on the other hand. In the present paper, we assume from the very beginning that the cut off procedure is already done and verify the existence of the global attractor for equation \eqref{eq_1D} just for completeness of the exposition.
\end{remark}

Let us fix the radius $R_0$ in such a way that $\mathcal A\subset B(R_0/2,0,\Phi)$ and introduce the set
\begin{equation}\label{2.inva}
\mathcal B:=\cup_{t\ge0}S(t)B(R_0,0,\Phi).
\end{equation}
Then, this set is bounded according to Theorem \ref{dis} and is invariant with respect to the semigroup $S(t)$:
\begin{equation}\label{2.bb}
\mathcal A\subset B(R_0/2,0,\Phi)\subset B(R_0,0,\Phi)\subset \mathcal B\subset B(\bar R,0,\Phi),\ \ S(t)\mathcal B\subset\mathcal B.
\end{equation}
Thus, we are not interested in the solutions starting outside of the set $\mathcal B$ and need to transform our equation on a set $\mathcal B$ only. From now on, we return to the scalar case $m=1$ and apply the transform $w=W(u)$ defined in the previous section. Recall that this transform depends on the parameter $K$. Moreover, according to \eqref{1.ul}, $W(\mathcal B)\subset B(C\bar R,0,\Phi)$ and, for all $K>K_0=K_0(\bar R)$, the inverse map $U=U(w)$ is well-defined and smooth on $B(2C\bar R,0,\Phi)$ (see Theorem \ref{Th1.main}. Then the transformed equation on $W(\mathcal B)$ reads
\begin{equation}\label{1D_fin}
\Dt w - \partial^2_x w+w +\<f(P_K(aw))\>\partial_x w= F_1(w)  + F_2(w),
\end{equation}
where
\begin{equation}\label{2.F1}
F_1(w) = (f(P_K(aw)) - f(aw)) \partial_x w
\end{equation}
and
\begin{equation}\label{2.F2}
F_2(w) = a^{-1}[\partial_x^2 a - \Dt a - f(aw)\partial_x a]w - a^{-1}g(aw).
\end{equation}
To obtain these formulas we just put $u(t,x):=a(t,x)w(t,x)$ in equation \eqref{eq_1D}, see also \eqref{0.trans}. However, in order to complete the transform, we need to express the function $a$ as well as $\Nx a$, $\Dx a$ and $\partial_t a$ in terms of the new variable $w$. Indeed, the map $w\to a(w)$ is defined as a solution of equation \eqref{a} (see Lemma \ref{Lem1.sol}). The derivative $\partial_x a$ is then can be found from equation \eqref{a}:
$$
(\partial_x a)(w)=\frac12(f(P_K(a(w)w))-\<f(P_K(a(w)w))\>)a(w).
$$
Differentiating this equation in $x$ and using it for evaluating $\Nx a$ in the differentiated equation, we get
$$
(\Dx a)(w)=\frac14(f(P_K(a(w)w))-\<f(P_K(a(w)w)\>)^2a(w)+\frac12 a(w)f'(P_K(a(w)w))\frac d{dx}(P_K(a(w)w)
$$
and since $P_K$ is a smoothifying operator, the terms $\partial_x a$ and $\partial_x^2 a$ can be expressed in a smooth way in terms of the map $w\to a(w)$. In particular, they are well defined on the ball $B(2C\bar R,0,\Phi)$ if $K>K_0$ and, due to the presence of derivatives $\frac d{dx} P_K(a(w)w)$, the norms of these operators and their Frechet derivatives depend on $K$.
\par
The term containing $\partial_t a$ is a bit more delicate since $a$ is local in time and we need to use the chain rule in order to find an expression for it. To do this, we first express the value $\partial_t P_Ku$ from equation \eqref{eq_1D}:
$$
P_K\Dt u=\Dx P_K(a(w)w)-P_K(a(w)w)-P_K(f(a(w)w)\partial_x(a(w)w))-P_Kg(a(w)w)
$$
and we see that the right-hand side smoothly expressed in terms of $w\to a(w)$. Therefore, the operator $w\to (P_K\Dt u)(w)$ is well-defined and smooth on the ball $B(2C\bar R,0,\Phi)$. Differentiating then the explicit formula \eqref{k} in time, we get
\begin{multline}\label{kk}
((\Dt a)(w))(x) =\\= a(w)(x)\frac12\int_{-\pi}^x f'(P_K(a(w)(s)w(s)))(P_K\Dt u)(w)(s) - \<f'(P_K(a(w)w))(P_K\Dt u)(w)\>\,ds
\end{multline}
and this shows that the map $w\to (\Dt a)(w)$ is also well-defined and smooth on $B(2C\bar R,\Phi)$. Thus, we have proved the following result.
\begin{lemma}\label{Lem2.F2} Under the above assumptions the map $F_2(w)$ is well-defined and smooth as the map from $B(2C\bar R,0,\Phi)$ to $\Phi$ for all $K\ge K_0$. In particular,
\begin{equation}\label{2.smF2}
\|F_2\|_{C^1(B(2C\bar R,0,\Phi),\Phi)}\le C_K,
\end{equation}
where the constant $C_K$ depends on $K\ge K_0$.
\end{lemma}
We now turn to the nonlinearity $F_1$. Obviously, it is well-defined and smooth as the map from $B(2C\bar R,0,\Phi)$ to $L^2_{per}(-\pi,\pi)$. Moreover, this map is  {\it small} if $K$ is large and this property is crucial for us.
\begin{lemma}\label{Lem2.F1} Under the above assumptions, the map $F_1(w)$ defined by \eqref{2.F1} is well-defined as a map from $B(2C\bar R,0,\Phi)$ to $L^2_{per}(-\pi,\pi)$ and the following estimate holds:
\begin{equation}\label{2.F1sm}
\|F_1\|_{C^1(B(2C\bar R,0,\Phi),L^2_{per})}\le CK^{-1/2},
\end{equation}
where the constant $C$ is independent of $K\ge K_0$.
\end{lemma}
Indeed, this estimate can be obtained arguing exactly as in \eqref{1.small}, see also \cite{KZI} for more details.
\par
The considered terms $F_1(w)$ and $F_2(w)$ in equation \eqref{1D_fin} are similar to the case of Dirichlet boundary conditions considered in \cite{KZI}. However, the extra term
$$
F_3(w):= \<f(P_K(aw))\>\partial_x w
$$
is specific to the case of periodic boundary conditions and is essentially different. Indeed, as before, we obviously have the smoothness of this term and the estimate
\begin{equation}\label{2.F3}
\|F_3\|_{C^1(B(2C\bar R,0,\Phi),L^2_{per})}\le C,
\end{equation}
where the constant $C$ is independent of $K\ge K_0$, but this constant {\it is not small} as $K\to\infty$, so we cannot treat this term as a perturbation.
\par
We finally note that the nonlinearities $F_i$ are defined not on the whole space $\Phi$, but only on a large ball $B(2C\bar R,0,\Phi)$ which is not convenient for constructing the inertial manifolds. To overcome this problem, we introduce a smooth cut off function $\theta\in C^\infty_0(\R)$ such that
$$
\theta(z)\equiv 1,\ \ |z|^2\le C\bar R \text{ and }\  \theta(z)\equiv0,\ \ |z|^2\ge 2C\bar R
$$
and the modified operators
\begin{equation}\label{2.op}
\mathcal F_i(w):=\theta(\|w\|^2_{H^1})F_i(w),\ \ i=1,2,3.
\end{equation}
Then the operators $\mathcal F_i$ are defined and smooth already in the whole phase space $\Phi$, coincide with $F_i$ on the ball $B(C\bar R,0,\Phi)$ and vanish outside of the ball $B(2C\bar R,0,\Phi)$. Moreover, the operator $\mathcal F_3(w)=\Theta(w)\partial_x w$, where
$$
\Theta(w):=\theta(\|w\|_{H^1}^2)\<f(P_K(a(w)w))\>
$$
is a smooth map from $\Phi$ to $\R$ which also vanishes outside of $B(2C\bar R,0,\Phi)$. Thus the transformed equation now reads
\begin{equation}\label{2.trans}
\Dt w-\Dx w+w+\Theta(w)\Nx w=\mathcal F_1(w)+\mathcal F_2(w).
\end{equation}
Moreover, this equation coincides with \eqref{1D_fin} on the ball $B(C\bar R,0,\Phi)$ and, consequently, the diffeomorphism $W:u\to w$ maps solutions of the initial equation \eqref{eq_1D} from some neighbourhood of the attractor $\mathcal A$ into the solutions of \eqref{2.trans} belonging to some neighbourhood of $W(\mathcal A)$. In particular, the set $W(\mathcal A)$ is an attractor for equation \eqref{2.trans} (maybe local since we do not control the behavior of solutions of \eqref{2.trans} outside of the ball $B(C\bar R,0,\Phi)$ where some new limit trajectories may a priori appear). Thus, from now on we forget about the initial equation \eqref{eq_1D} and will work with the transformed equation \eqref{2.trans} only. For the convenience of the reader, we collect the verified properties of maps $\mathcal F_i$ in the next theorem.

\begin{theorem}\label{T_F_1,F_2}
The operators $\mathcal F_1$, $\mathcal F_2$  and $\Theta$ belong to $C^\infty(\Phi,L^2_{per})$, $C^\infty(\Phi,\Phi)$ and $C^\infty(\Phi,\R)$
respectively and vanish   outside of a big ball $B(2C\bar R,0,\Phi)$. Moreover, the following estimates hold:
\begin{equation}\label{2.key}
\|\mathcal F_1\|_{C^1(\Phi,L^2_{per})}\le CK^{-1/2},\ \ \|\mathcal F_2\|_{C^1(\Phi,\Phi)}\le C_K,\ \ \|\Theta\|_{C^1(\Phi,\R)}\le C,
\end{equation}
where the constant $C$ is independent of $K$ and the constant $C_K$ may depend on $K$.
\end{theorem}
\begin{remark}\label{Rem2.per} We see that, in contrast to the case of Dirichlet boundary conditions considered in \cite{KZI}, in the periodic case the transform $W$ does not allow us to make the nonlinearity which contains spatial derivatives small, but makes it small up to the operator $\Theta(w)\Nx w$ only. Although this term has a very simple structure, it prevents us from using the standard Perron method of constructing the inertial manifolds and makes the situation essentially more complicated. Actually, overcoming this difficulty is one of two main results of the paper.
\end{remark}

\section{Scalar case: existence of an Inertial Manifold}\label{s3}

In this section, we will construct the inertial manifold for the transformed equation \eqref{2.trans}. To be more precise, in contrast to the case of Dirichlet boundary conditions, we do not know how to construct the inertial manifold directly for equation \eqref{2.trans} and need to introduce one more cut off function. We first note that arguing exactly as in Theorems \ref{dis} and \ref{Th2.attr}, we may prove that equation \eqref{2.trans} is uniquely solvable for every $w(0)\in\Phi$ and the corresponding solution $w(t)$ satisfies all of the estimates derived in Theorem \ref{dis}. This in turn means that the solution semigroup $S_{tr}(t):\Phi\to\Phi$ is well-defined, dissipative and possesses a global attractor $\mathcal A_{tr}\in H^2_{per}(-\pi,\pi)$. Moreover, according to the analogue of the $H^2$-dissipative estimate \eqref{2.h2dis}, the set
\begin{equation}
\mathcal B_{H^2}:=\cup_{t\in\R_+}S_{tr}(t)B(r,0,H^2_{per})
\end{equation}
will be invariant, bounded in $H^2_{per}$ set which contains the global attractor $\mathcal A_{tr}$:
\begin{equation}
S_{tr}(t)\mathcal B_{H^2}\subset \mathcal B_{H^2},\ \ \mathcal A_{tr}\subset \mathcal B_{H^2},\ \ \|\mathcal B_{H^2}\|_{H^2_{per}}\le \frac R2,
\end{equation}
where $r$ is large enough and $R>r$ is some number depending only on $r$. We also recall that by the construction of the transformed equation \eqref{2.trans} $W(\mathcal A)\subset \mathcal A_{tr}$ (where $\mathcal A$ is the attractor of the initial equation \eqref{eq_1D}) and
\begin{equation}\label{3.equiv}
S_{tr}(t)=W\circ S(t)\circ W^{-1}
\end{equation}
in a neighboorhood of the set $W(\mathcal A)$. Thus, the dynamics generated by  equation \eqref{2.trans} outside of the ball $B(R,0,H^2_{per})$ becomes not essential and we may change it there in order to simplify the construction of the inertial manifold. To this end, we introduce one more cut-off function $\phi\in C^\infty(\R)$ which is monotone decreasing and
\begin{equation}\label{3.catphi}
\phi(z)\equiv 0,\ \ z\le R^2,\ \ \phi(z)\equiv-\frac12,\ \ z\ge(2R)^2
\end{equation}
and one more nonlinear operator
\begin{equation}
T(w):=\phi(\|(\Dx-1)P_Nw\|^2_{L^2_{per}})(\Dx-1)P_Nw,
\end{equation}
where the number $N$ will actually coincide with the dimension of the manifold and will be fixed below. The key properties of this map are collected in the next lemma.
\begin{lemma}\label{T} The map $T$ is a $C^\infty$-smooth map from $\Phi$ to $P_N\Phi$. Moreover, its Frechet derivative $T'(w)$ is globally bounded as a map from $\Phi$ to $\mathcal L(\Phi,\Phi)$ and satisfies the following inequalities:
\begin{equation}
(T'(w)\xi, (\partial_x^2-1) P_N \xi) \le0,
\end{equation}
for all $w\in \Phi$   and
\begin{equation}
(T'(w)\xi, (\partial_x^2-1) P_N \xi) = -\frac12\|P_N \xi\|^2_{H^2}
\end{equation}
for $w\in \Phi$ such that $\|(\Dx-1) P_N w\|^2_{L^2_{per}} \ge 2R$.
\end{lemma}
\begin{proof} Indeed, the Frechet derivative of $T$ reads
\begin{multline}
T'(w)\xi=\phi(\|(\Dx-1)P_Nw\|^2_{L^2_{per}})(\Dx-1)P_N\xi+\\+
2\phi'(\|(\Dx-1)P_Nw\|^2_{L^2_{per}})((\Dx-1)P_N w,(\Dx-1)P_N\xi)(\Dx-1)P_Nw.
\end{multline}
Using the fact that $\phi'(z)=0$ for $z>4R^2$, we see that the derivative $T'(w)$ is uniformly bounded as a map from $\Phi$ to $\mathcal L(\Phi,\Phi)$ and, in particular, the map $w\to T(w)$ is globally Lipschitz as the map from $\Phi$ to $\Phi$. Moreover, since $\phi(z)\le0$ and $\phi'(z)\le0$ for all $z\in \Bbb R$, we have
\begin{multline}
(T'(w)\xi, (\partial_x^2-1) P_N\xi)= 2 \phi'(\|(\Dx-1)P_N w\|^2_{L^2})((\partial_x^2-1) P_N w, (\partial_x^2-1) P_N \xi)^2  + \\+ \phi(\|(\Dx-1)P_N w\|^2_{L^2}) \|(\partial_x^2-1) P_N \xi\|^2_{L^2} \le0.
\end{multline}
For the case $\|(\Dx-1)P_N w\|^2_{L^2} \ge 4R^2$ by definition $T(w) = -\frac{1}{2}(\partial_x^2-1) P_N w$ and consequently
\begin{equation}
(T'(w)\xi,(\partial_x^2-1) P_N\xi ) = -\frac 12\|(\Dx-1)P_N\xi\|^2_{L^2}
\end{equation}
and the lemma is proved.
\end{proof}
Thus, we arrive at the following final equation for the inertial manifold to be constructed:
\begin{equation}\label{1D_fin_fin}
\Dt w- \partial^2_x w+w+\Theta(w)\Nx w = T(w)+\mathcal F_1(w)+\mathcal F_2(w) .
\end{equation}
Note that this equation can be interpreted as a particular case of an abstract semilinear parabolic equation
\begin{equation}\label{3.abs}
\Dt w+Aw=\mathcal F(w)
\end{equation}
in a Hilbert space $\Phi:=H^1_{per}(-\pi,\pi)$, where $A:=1-\Dx$ (is a self-adjoint positive operator in $\Phi$ with compact inverse) and
$$
\mathcal F(w):=\mathcal F_1(w)+\mathcal F_2(w)+T(w)-\Theta(w)\Nx w.
$$
Indeed, as follows from Theorem \ref{T_F_1,F_2} and Lemma \ref{T}, the nonlinearity $\mathcal F$ is globally Lipschitz continuous as the map from $\Phi$ to $L^2_{per}(-\pi,\pi)=D(A^{-1/2})$. This, in particular, implies that this equation is also globally well-posed in $\Phi$, generates a dissipative semigroup $\bar S(t):\Phi\to\Phi$ and the corresponding solution $w(t)$ satisfies all of the estimates stated in Theorem \ref{dis}. Moreover, due to Theorem \ref{T_F_1,F_2} and the obvious fact that $Q_NT(w)=0$, the $Q_N$-component of the nonlinearity $\mathcal F$ is globally bounded:
\begin{equation}\label{3.qbound}
\|Q_N\mathcal F(w)\|_{L^2_{per}}\le C,
\end{equation}
where the constant $C$ is independent of $N$ and $w$. This property gives  the control for the $Q_N$-component of the solution $w$ which is crucial for what follows.
\begin{lemma} Let the nonlinearity $\mathcal F$ satisfy \eqref{3.qbound}. Then, for any $\kappa\in(0,1)$, there exists a constant $R_\kappa>0$ (independent of $N$) such that, for any solution $w(t)$ of equation \eqref{3.abs} with $w(0)\in H^{2-\kappa}_{per}(-\pi,\pi)$, the following estimate holds:
\begin{equation}\label{3.qnest}
\|Q_Nw(t)\|_{H^{2-\kappa}_{per}}\le (\|Q_Nw(0)\|_{H^{2-\kappa}_{per}}-R_\kappa)_+e^{-\alpha t}+R_{\kappa},
\end{equation}
where $z_+:=\max\{z,0\}$ and the positive constant $\alpha$ is independent of $\kappa$, $t$, $N$ and $w$.
\end{lemma}
\begin{proof} Indeed, according to the variation of constants formula, $Q_Nw(t)$ satisfies
\begin{equation}
Q_Nw(t)=Q_Nw(0)e^{-At}+\int_0^te^{-A(t-s)}Q_N\mathcal F(w(s))\,ds.
\end{equation}
Taking the $H^{2-\kappa}_{per}$-norm to both sides of this equality and using that
$$
\|e^{-A(t-s)}\|_{\mathcal L(L^2_{per},H^{2-\kappa}_{per})}\le Ce^{-\alpha(t-s)}(t-s)^{-1+\kappa/2}
$$
for some positive $C$ and $\alpha$, we end up with the following estimate:
\begin{multline}\label{3.frac}
\|Q_Nw(t)\|_{H^{2-\kappa}_{per}}\le \|Q_Nw(0)\|_{H^{2-\kappa}_{per}}e^{-\alpha t}+\\+C\int_0^te^{-\alpha(t-s)}(t-s)^{-1+\kappa/2}\|Q_N\mathcal F(w(s))\|_{L^2_{per}}\,ds\le\\\le
\|Q_Nw(0)\|_{H^{2-\kappa}_{per}}e^{-\alpha t}+C_1\int_0^t\frac{e^{-\alpha(t-s)}}{(t-s)^{1-\kappa/2}}\,ds
\end{multline}
and the assertion of the lemma is a straightforward corollary of this estimate. Thus, the lemma is proved.
\end{proof}
The proved lemma shows that the sets
\begin{equation}\label{3.bkap}
\Bbb B_\kappa:=\{w\in H^{2-\kappa}_{per},\ \ \|Q_Nw\|_{H^{2-\kappa}_{per}}\le R_\kappa\}
\end{equation}
are invariant with respect to the semigroup $\bar S(t)$:
$$
\bar S(t)\Bbb B_\kappa\subset \Bbb B_{\kappa}.
$$
\begin{remark} The auxiliary operator $T(w)$ has been introduced in \cite{mal-par} in order to construct the inertial manifolds for reaction-diffusion equations in higher dimensions using the so-called spatial averaging method. On the one hand, since $T(w)=-\frac12P_N(\Dx-1) w$ if $\|P_Nw\|_{H^2_{per}}$ is  large,   this term shifts roughly speaking the first $N$-eigenvalues and makes the spectral gap large enough to treat the nonlinearity. So, this trick actually allows to check the cone property for the case where $\|P_Nw\|_{H^2_{per}}\le 2R$. On the other hand, together with the control \eqref{3.qnest}, this gives us the control of $H^{2-\kappa}$-norm in the estimates related with the cone property, see also \cites{KZ1,Zel2} and the proof of Theorem \ref{Th3.main} below.
\par
Mention also that by the construction of the nonlinearity $T$, equation \eqref{1D_fin_fin} coinsides
with \eqref{2.trans} in the neighbourhod of the attractor $\mathcal A_{tr}$.
\end{remark}
We are ready to verify the existence of the inertial manifold for the problem \eqref{1D_fin_fin}. For the convenience of the reader, we first recall the definition of an inertial manifold and the result which guarantees its existence.
\begin{definition}\label{Def3.IM}
A set $\mathcal{M}\in \Phi$ to be called an inertial manifold for  problem \eqref{3.abs} if it satisfies the following properties:
\par
1. $\mathcal{M}$ is strictly invariant under the action of the semigroup $\bar S(t)$, i. e. $\bar S(t)\mathcal{M} = \mathcal{M}$;
\par
2. $\mathcal{M}$ is a Lipschitz submanifold of $\Phi$ which can be presented as a graph of a Lipschitz continuous function $M:P_N \Phi\to Q_N \Phi$ for some $N \in
\Bbb N$, i.e.,
\begin{equation}
\mathcal{M} = \left\{ w_+ + M(w_+), w_+ \in P_N \Phi \right\}\ \  \text{ and }\ \ \ \|M(w^1_+) - M(w_+^2)\|_{\Phi} \le L_M\|w_+^1 - w_+^2\|_{\Phi};
\end{equation}
for some constant $L_M$;
\par
3. $\mathcal{M}$ possesses an exponential tracking property, i.e. for any solution $w(t)$, $t \ge 0$, of problem \eqref{3.abs} there exists a solution $\tilde w(t)$, $t\in\R$, belonging to $\mathcal{M}$ for all $t\in\R$ such that:
\begin{equation}
\| w(t) - \tilde w(t)\|_{\Phi}\le C e^{-\gamma t}\|w(0) - \tilde{w}(0)\|_{\Phi}
\end{equation}
for some positive constants $C$ and $\gamma$.
\end{definition}
The proof of the existence of an inertial manifold will be based on the invariant cone property and the graph transform method, see \cites{FST,mal-par,rom-man,Zel2} for more details. To introduce the invariant cone property convenient for our purposes, we introduce the following quadratic form
\begin{equation}
V(\xi) = \|Q_N \xi\|^2_{\Phi} - \|P_N \xi\|^2_{\Phi},\ \  \|z\|_{\Phi}^2:=(Az,z)=\|\Nx z\|^2_{L^2_{per}}+\|z\|^2_{L^2_{per}}
\end{equation}
and corresponding cone in the phase space $\Phi$:
\begin{equation}
\mathcal K^+ = \bigg\{\xi \in \Phi : V(\xi)\le 0 \bigg\}.
\end{equation}
\begin{definition}
We say that equation \eqref{3.abs} possesses a strong cone property in the differential form if there exist a positive constant $\mu$ and a bounded function $\alpha: \Phi \to \Bbb R$, which satisfies the property:
\begin{equation}
0<\alpha_- \le \alpha(w)\le \alpha_+ <\infty,
\end{equation}
such that for any solution $w(t)\in \Phi$, $t\in [0,T]$, of equation \eqref{3.abs} and any solution $\xi(t)$ of the corresponding equation in variations:
\begin{equation}\label{eq_var}
\Dt \xi +A \xi =  \mathcal F'(w(t))\xi
\end{equation}
the following inequality holds:
\begin{equation}\label{srt_cone}
\frac{d}{dt}V(\xi) + \alpha(w)V(\xi) \le -\mu \|\xi\|^2_{H^2_{per}}.
\end{equation}
If inequality \eqref{srt_cone} holds not for all trajectories $w(t)$, but only for the ones belonging to some invariant set, we will say that the strong cone property is satisfied on this set.
\end{definition}
The next theorem gives the conditions which guarantees the existence of the inertial manifold for the abstract equation \eqref{3.abs} which we need to verify for our case of equation \eqref{1D_fin_fin}.

\begin{theorem}\label{Th3.IM} Let the nonlinearity $\mathcal F$ be globally Lipschitz continuous as a map from $\Phi$ to $L^2_{per}=D(A^{-1/2})$ and let the number $N$ be chosen in such a way that  $Q_N\mathcal F$ is globally bounded on $\Phi$ and the strong  cone property in the differential form is satisfied on the invariant set $\Bbb B_\kappa$ for some $\kappa\in(0,1]$ defined by \eqref{3.bkap}. Then \eqref{3.abs} possesses a $(2N+1)$-dimensional Lipschitz inertial manifold in the space $\Phi$.
\par
Moreover, if the nonlinearity $\mathcal F$ is of class $C^{1+\beta}(\Phi,L^2_{per})$ for some $\beta>0$, then the inertial manifold is of class $C^{1+\eb}$ for some $\eb=\eb(\beta,N)>0$.
\end{theorem}

It is well-known result, see e. g. \cites{mal-par,Zel2}, that the validity of the strong cone property in the differential form leads to the existence of an $(2N+1)-$dimensional inertial manifold.

The following theorem can be considered as one of the main results of this chapter.
\begin{theorem}\label{Th3.main}
Under above assumptions  for infinity many values of $N\in \Bbb N$ equation \eqref{1D_fin_fin} possesses a $(2N+1)-$dimensional inertial manifold. Moreover these inertial manifolds are $C^{1+\varepsilon}$-smooth for some small positive $\eb=\eb(N)>0$.
\end{theorem}
\begin{proof} According to Theorem \ref{Th3.IM}, we only need to verify the validity of the strong cone condition on the invariant set $\Bbb B_\kappa$ for some $\kappa>0$. The rest of the assumptions of this theorem are already verified above. We fix $\kappa=\frac14$ and write out the equation of variations which corresponds to equation \eqref{1D_fin_fin}:
\begin{equation}\label{3.var}
\partial_t \xi  +A\xi = -\Theta(w)\Nx\xi-(\Theta'(w),\xi)_{H^1}\Nx w+\mathcal F_1'(w)\xi+\mathcal F_2'(w)\xi+T'(w)\xi,
\end{equation}
where  $w(t)$ is the solution of the equation \eqref{1D_fin_fin} belonging to $\Bbb B_{1/4}$.
Multiplying this equation
 by $A Q_N \xi - A P_N \xi$ and denoting $\bar\alpha:=\frac{\lambda_{2N+1}+\lambda_{2N}}2$, we get
\begin{multline} \label{h}
\frac{1}{2}\frac{d}{dt} V(\xi) + \bar\alpha V(\xi) =  ((\bar\alpha -A)\xi, A Q_N \xi) - ((\bar\alpha -A) \xi, AP_N \xi) -\\-\Theta(w)(\Nx\xi,A Q_N \xi - A P_N \xi)-(\Theta'(w),\xi)_{H^1}(\Nx w,A Q_N \xi - A P_N \xi)+\\+(\mathcal{F}'_1(w)\xi,A Q_N \xi - A P_N \xi )+
(\mathcal{F}'_2(w)\xi, A Q_N \xi -A P_N \xi)  - (T'(w)\xi, A P_N \xi).
\end{multline}
Let us estimate every term in the right-hand side of this inequality separately.
Integrating by parts in the first term, we see that
\begin{equation}
\Theta(w)(\partial_x \xi,A Q_N \xi - A P_N \xi) =0.
\end{equation}
Due to estimate \eqref{2.key} on the nonlinearity $\mathcal{F}_1$ we have
\begin{equation}\label{est_F_1}
(\mathcal{F}'_1(w)\xi, A Q_N \xi - A P_N \xi) \le C K^{-1/2}\|\xi\|_{H^1_{per}}\|\xi\|_{H^2_{per}},
\end{equation}
and estimate \eqref{2.key} on $\mathcal{F}_2$ gives us
\begin{equation}\label{est_F_2}
(\mathcal{F}'_2(w)\xi, A Q_N \xi -A P_N \xi) \le C_K \| \xi\|^2_{H^1_{per}}.
\end{equation}
In next estimates, we will use the notations
 \begin{multline}
 e_{2n}:=\cos(nx),\ n=\{0\}\cup\Bbb N,\ \ e_{2n-1}:=\sin(nx), \ n\in\Bbb N;\\
 \lambda_{0}=1,\ \ \lambda_{2n}=\lambda_{2n-1}:=n^2+1,\ \ n\in\Bbb N
 \end{multline}
 and  formulas
 $$
 \xi:=\sum_{n=1}^\infty\xi_n e_n,\ \ P_N\xi=\sum_{n=1}^{2N}\xi_n e_n,\ \ Q_N\xi:=\sum_{n=2N+1}^\infty\xi_n e_n.
 $$
 Then, we estimate the linear terms as follows
\begin{multline}\label{+}
((\bar\alpha -A)\xi, -A P_N \xi)=\sum_{n=0}^{2N}(\lambda_n^2 - \bar\alpha \lambda_n)\xi_n^2 = \frac{1}{2}\sum_{n=0}^{2N}(\lambda_n - \bar\alpha)\lambda_n\xi_n^2 + \\ \frac{1}{4}\sum_{n=0}^{2N}\left(\lambda_n^{3/4} - \frac{\bar\alpha}{\lambda_n^{1/4}}\right)\lambda_n^{5/4}\xi_n^2 +  \frac{1}{4} \sum_{n=0}^{2N}\left(1-\frac{\bar\alpha}{\lambda_n}\right)\lambda_n^2 \xi_n^2\le\\
\frac{1}{2}(\lambda_{2N} -\bar\alpha)\|P_N \xi\|_{H^1}^2 + \frac{1}{4}\left(\lambda_{2N}^{3/4} - \frac{\bar\alpha}{\lambda_{2N}^{1/4}}\right)\|P_N \xi\|_{H^{5/4}}^2+  \frac{1}{4}\left(1 - \frac{\bar\alpha}{\lambda_{2N}}\right)\|P_N \xi\|_{H^2}^2,
\end{multline}
and
\begin{multline}\label{-}
((\bar\alpha -A)\xi, A Q_N \xi) = \sum_{n=2N+1}^{\infty}(\bar\alpha \lambda_n - \lambda_n^2)\xi_n^2\le\\ \frac{1}{2}(\bar\alpha - \lambda_{2N+1})\|Q_N \xi\|_{H^1}^2 +\frac{1}{4}\left(\frac{\bar\alpha}{\lambda^{1/4}_{2N+1}} - \lambda_{2N+1}^{3/4}\right)\|Q_N \xi\|^2_{H^{5/4}}+ \frac{1}{4}\left(\frac{\bar\alpha}{\lambda_{2N+1}} - 1\right)\|Q_N \xi\|^2_{H^2}.
\end{multline}
We recall that $\bar\alpha = \frac{\lambda_{2N+1} + \lambda_{2N}}{2}$, consequently
\begin{multline}
((\alpha -A)\xi, -A P_N \xi) + ((\alpha -A)\xi, A Q_N \xi) \le\\\le -\frac{\lambda_{2N+1} - \lambda_{2N}}{4}\|\xi\|^2_{H^1}- \frac{\lambda_{2N+1}^{3/4} - \lambda_{2N}^{3/4}}{8}\|\xi\|^2_{H^{5/4}} - \frac{\lambda_{2N+1} - \lambda_{2N}}{8\lambda_{2N+1}}\|\xi\|^2_{H^2}=\\=-\frac{\lambda_{2N+1} - \lambda_{2N}}{4}\|\xi\|^2_{H^1}- \frac{\lambda_{2N+1}^{3/4} - \lambda_{2N}^{3/4}}{8}\|\xi\|^2_{H^{5/4}} - \frac{\lambda_{2N+1} - \lambda_{2N}}{16 \lambda_{2N+1}}\|\xi\|^2_{H^2}-\mu\|\xi\|^2_{H^2},
\end{multline}
where, we set $\mu:=\frac{\lambda_{2N+1} - \lambda_{2N}}{16 \lambda_{2N+1}}$. Inserting the obtained estimates into the right-hand side of \eqref{h} and using that
$$
C K^{-1/2}\|\xi\|_{H^1_{per}}\|\xi\|_{H^2_{per}}\le \frac{\lambda_{2N+1}-\lambda_{2N}}8\|\xi\|^2_{H^1_{per}}+C^2K^{-1}\frac2{\lambda_{2N+1} - \lambda_{2N}}\|\xi\|^2_{H^2_{per}}
$$
 we arrive at
\begin{multline}\label{3.h}
\frac{1}{2}\frac{d}{dt} V(\xi) + \bar\alpha V(\xi)+\mu\|\xi\|^2_{H^2_{per}}\le
-(\Theta'(w),\xi)_{H^1}(\Nx w,A Q_N \xi - A P_N \xi)-\\-(T'(w)\xi, A P_N \xi)-
\(\frac{\lambda_{2N+1}-\lambda_{2N}}8-C_K\)\|\xi\|^2_{H^1_{per}}
 - \frac{\lambda_{2N+1}^{3/4} - \lambda_{2N}^{3/4}}{8 }\|\xi\|^2_{H^{5/4}_{per}}-\\-\(\frac{\lambda_{2N+1} - \lambda_{2N}}{16 \lambda_{2N+1}}-C^2K^{-1}\frac2{\lambda_{2N+1} - \lambda_{2N}}\)\|\xi\|^2_{H^2_{per}}
\end{multline}
Let us now estimate the first term in the right-hand side of \eqref{3.h}. To this end, we fix an arbitrary $t \ge 0$ and consider two cases: 1)  $\|P_N w(t)\|_{H^2_{per}}\le2R$ and 2) $\|P_N w(t)\|_{H^2_{per}}>2R$ where the constant $R$ is the same as in \eqref{3.catphi}.
\par
In the first case, using also that $w\in\Bbb B_{1/4}$, we conclude that
\begin{equation}
\|w\|^2_{H^{7/4}_{per}} \le \|P_N w\|^2_{H^2_{per}} + \|Q_N w\|^2_{H^{7/4}_{per}} \le 2R + R_{1/4}:=\bar{C}.
\end{equation}
Therefore, using also that $\Theta'(w)$ is globally bounded in $H^1_{per}$
\begin{equation}
|(\Theta'(w),\xi)_{H^1}(\partial_x w,A Q_N \xi - AP_N \xi)|\le C\|w\|_{H^{7/4}_{per}}\|\xi\|_{H^1_{per}}\|\xi\|_{H^{5/4}_{per}}\le \tilde{C}\|\xi\|^2_{H^{5/4}_{per}}.
\end{equation}


As follows from Lemma \ref{T} additional term containing $T'(w)$ does not make any difference since
$(T'(w)\xi, -A P_N \xi) \le 0$.
Therefore, in the first case inequality \eqref{3.h} reads
\begin{multline}\label{3.h1}
\frac{1}{2}\frac{d}{dt} V(\xi) + \bar\alpha V(\xi)+\mu\|\xi\|^2_{H^2_{per}}\le
\(\frac{\lambda_{2N+1}-\lambda_{2N}}8-C_K\)\|\xi\|^2_{H^1_{per}}-\\
 - \(\frac{\lambda_{2N+1}^{3/4} - \lambda_{2N}^{3/4}}{8 }-\tilde C\)\|\xi\|^2_{H^{5/4}_{per}}-\(\frac{\lambda_{2N+1} - \lambda_{2N}}{16 \lambda_{2N+1}}-C^2K^{-1}\frac2{\lambda_{2N+1} - \lambda_{2N}}\)\|\xi\|^2_{H^2_{per}}.
\end{multline}
We now recall that the  eigenvalues  $\lambda_{2N}=N^2+1$ and $\lambda_{2N+1}=(N+1)^2+1$. Therefore, for  $N>0$, $\lambda_{2N+1}-\lambda_{2N}=2N+1$ and
$$
\frac{(\lambda_{2N+1}-\lambda_{2N})^2}{\lambda_{2N+1}}=\frac{(2N+1)^2}{(N+1)^2+1}\ge1.
$$
Thus, if we fix the parameter $K\ge K_0$ in such way that
 \begin{equation}\label{3.K}
 C^2K^{-1}\le \frac1{64},
 \end{equation}
 the last term in the left hand side will be non-positive. Crucial for us that we may fix $K$ in such way that this property holds for all  $N$s simultaneously. Obviously, the first two terms in the LHS of \eqref{3.h1} will be also non-positive if $N$  is large enough.
Thus in the case $\|P_N w\|_{H^2_{per}}\le 2R$, we may take
$$
\alpha(w):=\bar \alpha=\frac{\lambda_{2N+1}+\lambda_{2N}}2
$$
and the strong cone condition will be satisfied.
\par
Let us now consider the second case where $\|P_Nw(t)\|_{H^2_{per}}>2R$. In this case, the auxiliary map $T$ is really helpful. Indeed, according to Lemma \ref{T},
\begin{multline}
(T'(w)\xi,-AP_N\xi)\le -\frac12\|AP_N\xi\|_{L^2_{per}}^2\le -\frac{\lambda_{2N}}2\|P_N\xi\|^2_{H^1_{per}}=-\frac{\lambda_{2N}}4\|P_N\xi\|^2_{H^1_{per}}+\\+
\frac{\lambda_{2N}}4\(\|Q_N\xi\|^2_{H^1_{per}}-\|P_N\xi\|^2_{H^1_{per}}\)-
\frac{\lambda_{2N}}4\|Q_N\xi\|^2_{H^1_{per}}=\frac{\lambda_{2N}}4V(\xi)-\frac{\lambda_{2N}}4
\|\xi\|^2_{H^1_{per}}.
\end{multline}
Using now that the norm of the derivative $\|\Theta'(w)\|_{H^1_{per}}$ vanishes if $\|w\|_{H^1_{per}}$ is large and, consequently, $\|\Theta'(w)\|_{H^1_{per}}\|w\|_{H^1_{per}}\le \bar C_1$, we estimate the first term in the right-hand side of \eqref{3.h} as follows:
\begin{multline}
|(\Theta'(w),\xi)_{H^1}(\Nx w,AQ_N\xi-AP_{N}\xi)|\le \|\Theta'(w)\|_{H^1_{per}}\|w\|_{H^1_{per}}\|\xi\|_{H^1_{per}}\|\xi\|_{H^2_{per}}\le\\\le \bar C_1\|\xi\|_{H^1_{per}}\|\xi\|_{H^2_{per}}\le \bar C_1^2\frac{8\lambda_{2N+1}}{\lambda_{2N+1}-\lambda_{2N}}\|\xi\|^2_{H^1_{per}}+
\frac{\lambda_{2N+1}-\lambda_{2N}}{32\lambda_{2N+1}}\|\xi\|^2_{H^2_{per}}.
\end{multline}
Thus, the analogue of \eqref{3.h1} for the second case reads
\begin{multline}\label{3.h2}
\frac{1}{2}\frac{d}{dt} V(\xi) + \(\bar\alpha-\frac{\lambda_{2N}}4\) V(\xi)+\mu\|\xi\|^2_{H^2_{per}}\le
\(\frac{\lambda_{2N+1}-\lambda_{2N}}8-C_K\)\|\xi\|^2_{H^1_{per}}-\\
 - \(\frac{\lambda_{2N}}4 -\bar C_1^2\frac{8\lambda_{2N+1}}{\lambda_{2N+1}-\lambda_{2N}}\)\|\xi\|^2_{H^{1}_{per}}-\(\frac{\lambda_{2N+1} - \lambda_{2N}}{32 \lambda_{2N+1}}-C^2K^{-1}\frac2{\lambda_{2N+1} - \lambda_{2N}}\)\|\xi\|^2_{H^2_{per}}.
\end{multline}
We see that the third term in the right-hand side is non-positive if the parameter $K$ satisfies exactly the same assumption \eqref{3.K} as in the first case (in particular, it can be fixed independently of $N$). Moreover, since $\lambda_{2N}\sim N^2$ and $\lambda_{2N+1}-\lambda_{2N}\sim 2N+1$, the second term and the first terms are also non-positive if $N$ is large enough. Thus, we are able to fix the parameters $K$ and $N$ in such ways that the right-hand sides of {\it both} inequalities \eqref{3.h1} and \eqref{3.h2} will be non-positive. Let us now introduce the function
\begin{equation}\label{3.a}
\alpha(w):=\begin{cases}\frac{\lambda_{2N+1}+\lambda_{2N}}2,\ &\ \|P_Nw\|_{H^2_{per}}\le 2R\\
   \frac{\lambda_{2N+1}+\lambda_{2N}}2-\frac{\lambda_{2N}}4,\ & \ \|P_Nw\|_{H^2_{per}}>2R.
    \end{cases}
\end{equation}
Then, we have proved that for all sufficiently large  $N$s, the strong cone inequality
\begin{equation}\label{3.cone}
\frac12\frac d{dt}V(\xi(t))+\alpha(w(t))V(\xi(t))\le-\mu\|\xi(t)\|^2_{H^2_{per}}
\end{equation}
is satisfied and the theorem is proved.
\end{proof}

\section{Vector case: a counterexample}\label{s4}

In this section, we will show that, in contrast to the scalar case considered above, the (Lipschitz continuous) IM may not exist in the case of a system of RDA equations \eqref{eq_1D} (i.e., if $m>1$). Analogously to \cite{EKZ}, our counterexample is built up based on the counterexample to Floquet theory for {\it linear} RDA equations with time-periodic coefficients. Namely, we consider the following system of linear RDA equations:
\begin{equation}\label{au}
\Dt{\rm u}-\partial_x^2{\rm u}+f(t,x)\partial_x{\rm u}+g(t,x){\rm u}=0
\end{equation}
endowed with periodic boundary conditions. We assume that ${\rm u}=(v(t,x),u(t,x))$ where the unknown functions as well as given $2T$-periodic in time functions $f$ and $g$ are {\it complex} valued, so we will consider a system of two coupled complex valued RDA equations. Of course, separating the real and imaginary parts of functions $v$ and $u$, we nay rewrite it as a system of four real-valued RDA equations with respect to ${\rm u}=(v_{Re},v_{Im},u_{Re},u_{Im})$, but preserving the complex structure is more convenient for our purposes. The main idea is to construct the functions $f$ and $g$ in such a way that {\it all} solutions ${\rm u}(t)$ will decay faster than exponentially as $t\to\infty$. If these functions are constructed, the standard trick with producing the space-time periodic functions $f$ and $g$ as particular solutions of some extra nonlinear RDA system, will give us a super-exponentially attracting limit cycle inside of the global attractor (see e.g., \cite{EKZ}) which clearly contradicts the existence of the IM for the full system.
\par
We first recall that, at least for smooth functions $f$ and $g$, equation \eqref{au} is well-posed in the phase space $\Phi$ (this can be established analogously to Theorem \ref{dis}) and generates a dissipative dynamical process $\{U(t,\tau),\ t\ge\tau, \ \tau\in\R\}$ in the phase space $\Phi$ via
\begin{equation}\label{4.U}
U(t,\tau){\rm u}_\tau:={\rm u}(t), \ \ U(t,s)=U(t,\tau)\circ U(\tau,s),\ \ t\ge\tau\ge s,
\end{equation}
where the function ${\rm u}(t)$ solves \eqref{au} with the initial data
 ${\rm u}\big|_{t=\tau}={\rm u}_\tau\in\Phi$.
In particular, since the functions $f$ and $g$ are $2T$-periodic in time, the long-time behavior of solutions of \eqref{au} is completely determined by the iterations of the period map
\begin{equation}\label{4.per}
P:=U(2T,0).
\end{equation}
Since, due to the smoothing property, the linear operator $P$ is compact its spectrum consists of $\{0\}$ as an essential spectrum and at most countable number of non-zero eigenvalues of finite multiplicity. It is well-known that any eigenvalue $\mu\ne0$ of this operator generates the so-called Floquet-Bloch solutions of \eqref{au} of the form
$$
{\rm u}_{\mu,n}(t):=t^{n-1}e^{\nu t}Q_{n-1}(t),
$$
 where $\nu:=\frac1{2T}\ln\mu$, $n\ge1$ does not exceed the algebraic multiplicity of the eigenvalue $\mu$ and $Q_{n}(t)$ are $2T$-periodic $\Phi$-valued functions. It is also known that, at least on the level of abstract parabolic equations in a Hilbert space, the linear combinations of Floquet-Bloch solutions are not dense in the space of all solutions. Moreover, the point spectrum of the operator $P$ may be empty which means that
 \begin{equation}\label{4.spec}
 \sigma(P)=\{0\},
 \end{equation}
see \cite{kuch} for more details. According to the Gelfand spectral radius formula, this will be the case when {\it all} solutions of problem \eqref{au} decay faster than exponential as $t\to\infty$ and this is exactly the case which we are interested in. The next theorem gives the desired example of the functions $f$ and $g$ such that \eqref{4.spec} is satisfied. To the best of our knowledge, similar examples have been previously known only for abstract parabolic equations (with non-local nonlinearities), but not for systems of second order parabolic PDEs.

\begin{theorem}\label{TH2} For every sufficiently large $T$ there exist smooth
  functions $f(t,x)\in \mathcal L(\Bbb C^2,\Bbb C^2)$ and $g(t,x)\in \Bbb C^2$ which are $2T$-periodic in time and $2\pi$-periodic in space such that all solutions of equation \eqref{au} decay faster than exponential as $t\to\infty$. Moreover, the following estimate holds for any of such solutions
\begin{equation} \label{dec}
\|{\rm u}(t)\|_{L^2} \le C e^{-\gamma t^3} \|{\rm u}(0)\|_{L^2},
\end{equation}
where positive constants $C$ and $\gamma$ are independent of ${\rm u}(0)\in L^2(-\pi,\pi)$.
\end{theorem}

\begin{proof} Let now $e_n:=e^{inx}$, $n\in\Bbb Z$ be the eigenvectors of the operator $-\Dx$ acting in the space of complex-valued functions. Obviously the corresponding eigenvalues are $\lambda_n=n^2$. The following simple formula is however crucial for the construction of our counterexample:
\begin{equation}
e_{n+1}=e^{ix}e_n, \ \ (\Dx +2i\Nx-1)e_n=-\lambda_{n+1}e_n, \ \ n\in\Bbb Z.
\end{equation}
Keeping in mind that our equation has two components ${\rm u}=(v,u)$, we introduce the following base vectors in $[L^2_{per}(-\pi,\pi)]^2$:
\begin{equation}\label{4.vec}
e_n^v:=\(\begin{matrix}1\\0\end{matrix}\)e_n,\ \ e_n^u:=\(\begin{matrix}0\\1\end{matrix}\)e_n.
\end{equation}
Then, the vectors $\{e_n^v,e_n^u\}_{n\in\Bbb Z}$ form an orthogonal base in the space $[L^2_{per}(-\pi,\pi;\Bbb C)]^2$. Moreover, these are the eigenvectors for the unperturbed problem \eqref{au} (with $f=g=0$) which correspond to the eigenvalue $\lambda_n=n^2$ and for every $n\ne0$, the corresponding eigenspace is spanned by $\{e_{\pm n}^v,e^u_{\pm n}\}$ and therefore has the complex dimension $4$ (and real dimension $8$). We intend to construct the functions $f$ and $g$ in such a way that the corresponding period map has the following properties:
\begin{equation}\label{P}
P e_n^v=\mu_ne_{n+1}^v,\ \ P e_n^u=\nu_n e_{n-1}^u,\ \ n\in \Bbb N,
\end{equation}
where  $-\mu_n$ and $-\nu_n$ are some positive multipliers such that
 $$
 |\mu_n|+|\nu_n| \le e^{-KT n^2}
 $$
 for some $K>0$ independent of $n$. Indeed, assume that such example is constructed. Then, clearly the point spectrum of $P$ is empty and, moreover, the following estimate holds:
\begin{equation}
\|P^Ne_{n}^v\|_{L^2} \le e^{-KT\sum_{k=n}^{n+N}k^2} = e^{-\frac{KT}6(N+1)(2N^2+6Nn+6n^2+N)}\le Ce^{-\gamma  N^3},
\end{equation}
for some $\gamma>0$ which is independent of $N$ and $n$ (here we have implicitly used the positivity of the quadratic form $2N^2+6Nn+6n^2$). Arguing analogously, we have
\begin{equation}
\|P^N e_n^u\|_{L^2}\le e^{-KT\sum_{k=n-N}^n k^2}\le Ce^{-\gamma N^3}.
\end{equation}
These estimates, together with \eqref{P}, imply that
\begin{equation}\label{4.cube}
\|P^N\|_{\mathcal L(L^2,L^2)}\le Ce^{-\gamma N^3}.
\end{equation}
Thus, estimate \eqref{dec} is verified and we only need to construct the functions $f$ and $g$ for which the period map $P$ of equation \eqref{au} will satisfy \eqref{P}. Roughly speaking, similarly to \cite{EKZ}, we initially take the unperturbed equation
$$
\Dt{\rm u}=\Dx {\rm u},
$$
and split the time interval  $[0,2T]$ on two parts $[0,T]$ and $[T,2T]$. At the first interval, we shift the spectrum of the $v$-component by adding the term $2i\Nx v-v$ (after this shift the vectors $e_n^v$ and $e_{n+1}^u$ will be in the eigenspace which corresponds to the eigenvalue $\lambda_{n+1}$) and switch on the "rotation" in the plane spanned by $\{e_n^v,e_{n+1}^u\}$ by adding the proper anti-symmetric term. This guesses the following form of the perturbed equation:
$$
\Dt v=\Dx v+(2i\Nx v-v)-\eb e^{-ix}u,\ \ \Dt u=\Dx v+\eb e^{ix}v,\ \ t\in[0,T],
$$
on the first half-period. The parameter $\eb>0$ should be chosen in such way that the half-period map $U(T,0)$ will rotate the direction of $e_n^v$ into the direction of $e_{n+1}^u$ and vise versa.
\par
At the second half-period, we need not to do shift and just put the "rotation" terms
$$
\Dt v=\Dx v-\eb u,\ \ \Dt u=\Dx u+\eb v,\ \ t\in[T,2T],
$$
where we again chose $\eb>0$ in such way that the half-period map $U(2T,T)$ rotates the direction of $e^v_n$ to the direction of $e_n^u$ and vise versa. Then, as not difficult to see  the composition $P=U(2T,T)\circ U(T,0)$ will satisfy relations \eqref{P} and the estimates for $\mu_n$ and $\nu_n$ will be also satisfied.
\par
Thus, the above arguments allow us to construct the desired counterexample in the class of piecewise  constant (in time) periodic functions $f$ and $g$. However, in order to build the counterexample to inertial manifolds, we need the functions $f$ and $g$ to be {\it smooth}, so we need to "smoothify" our construction by adding the properly chosen cut-off functions.
\par
Namely, let us fix an auxiliary $2T-$periodic function $y(t)$ satisfying the following assumptions:
\begin{equation}\label{ass}
\begin{array}{l}
1.\ \ y(t) \text{ is odd and } y(T-t) = y(t)\text{ for all }t;\\
2.\ \ y(t)\text{ has a maximum point at } t=T/2\text{ and }y(T/2) = 1;\\
3.\ \ y''(t)\le 0 \text{ for } 0<t<T \text{ and }y'(t)> 0 \text{ for }0<t<T/2.
\end{array}
\end{equation}
One of the possible choices of  $y(t)$ is  $\sin (\pi t/ T)$.
In addition we introduce a pair of smooth non-negative cut-off functions $\theta_1$ and $\theta_2$:
\begin{equation}
 \theta_1(y) = 0 \text{, for }y \le 1/4,\ \theta_1(y)=1 \text{, for }y \ge 1/2;
\end{equation}
\begin{equation}
\theta_2(y) = 0 \text{, for }y \le 0,\ \theta_2(y)=1 \text{, for } y \ge 1/4.
\end{equation}
Now we are ready to introduce the desired equations:
\begin{equation} \label{exp_au}
\begin{cases}
\pt v = \Dx v + (2 i \Nx v - v)\theta_2(y) - \varepsilon e^{- i x} u  \theta_1(y) - \varepsilon u \theta_1(-y),\\
\pt u = \Dx u + \varepsilon e^{ix} v \theta_1(y) + \varepsilon v \theta_1(-y),
\end{cases}
\end{equation}
where $\varepsilon$ is a small parameter which will be choosing later in such a way that on the first half-period
\begin{equation}
U(T,0)e_n^v = K_n^+e_{n+1}^u \ \ \text{ and } \ \
U(T,0)e_n^u = C_n^+e_{n-1}^v,
\end{equation}
and on the other part of period
\begin{equation}
U(2T,T)e_{n+1}^u =K_n^-e_{n+1}^v \ \ \text{ and }\ \
U(2T,T)e_{n-1}^v =C_n^-e_{n-1}^u,
\end{equation}
for some contraction factors $K^+_n, C^+_n$,  $K^-_n, C^-_n$.
\par
We claim that the proposed equations satisfy all the assumptions of the theorem. Indeed, let us first consider equations \eqref{exp_au} on a half-period $[0,T]$ which due to the specific form of the cut-off functions $\theta_1(y)$, $\theta_2(y)$ and time-periodic function $y(t)$ have a form
\begin{equation}\label{exp_au_1}
\begin{cases}
\pt v = \Dx v + (2 i \Nx v - v)\theta_2(y) - \varepsilon e^{- i x} u  \theta_1(y),\\
\pt u = \Dx u + \varepsilon e^{ix} v \theta_1(y).
\end{cases}
\end{equation}
We fix $T_0$ such that $y(T_0)=1/4$. Then on the intervals $[0, T_0]$ and $[T- T_0, T]$ equations \eqref{exp_au_1} become decoupled:
\begin{equation}
\begin{cases}
\pt v = \Dx v + (2 i \Nx v - v)\theta_2(y),\\
\pt u = \Dx u.
\end{cases}
\end{equation}
Writing these equations in Fourier coordinates, we obtain
\begin{equation}
\begin{cases}
\frac{d}{dt}v_n = -(n^2 + (2 n +1)\theta_2(y))v_n,\\
\frac{d}{dt}u_n = - n^2 u_n.
\end{cases}
\end{equation}
Therefore,
\begin{equation*}
U(T_0,0) e_{n}^v =e^{-T_0n^2 - (2n+1)\int_0^{T_0}\theta_2(y(t))dt }  e_{n}^v, \ \
U(T_0,0) e_{n}^u =e^{-T_0n^2}  e_{n}^u,
\end{equation*}
and
\begin{equation*}
U(T,T-T_0) e_{n}^v =e^{-T_0n^2 - (2n+1)\int^{T}_{T - T_0}\theta_2(y(t))dt }  e_{n}^v,\ \
U(T,T -T_0) e_{n}^u =e^{-T_0n^2}  e_{n}^u.
\end{equation*}
Let us turn to the map $U(T - T_0, T_0)$.
The specific choice of the cut-off functions allows us to rewrite the equation \eqref{exp_au_1} on this interval in the form
\begin{equation}
\begin{cases}
\pt v = \Dx v + 2 i \Nx v - v - \varepsilon e^{- i x} u  \theta_1(y),\\
\pt u = \Dx u + \varepsilon e^{ix} v \theta_1(y).
\end{cases}
\end{equation}
Since $e_n = e^{inx}$ and consequently $e_{n+1} = e^{ix}e_n$, after writing down our equations in Fourier modes an equation on $v_n$ will be coupled with an equation on $u_{n+1}$:
\begin{equation} \label{T_0T}
\begin{cases}
\frac{d}{dt}v_n = - (n+1)^2 v_n - \varepsilon u_{n+1}\theta_1(y),\\
\frac{d}{dt}u_{n+1} = -(n+1)^2 u_{n+1} + \varepsilon v_n \theta_1(y).
\end{cases}
\end{equation}
To study these equations we introduce the polar coordinates:
\begin{equation}
v_n + i u_{n+1} = R_n e^{i \phi_n},
\end{equation}
which leads to two separate equations on the radial and angular coordinates:
\begin{equation}
\frac{d}{dt}R_n = -(n+1)^2 R_n, \ \
\frac{d}{dt} \phi_n = \varepsilon \theta_1(y(t)).
\end{equation}
Fixing
\begin{equation}\label{eps}
\varepsilon:=\frac{\pi}{2 \int_{T_0}^{T - T_0}\theta_1(y(t))dt},
\end{equation}
we see that $U(T- T_0, T_0)$ restricted on the $ \operatorname{span}\left\{ e_{n}^v, e_{n+1}^u \right\}  $ is a composition of the rotation on the angle $\pi/2$ and the proper contraction, more precisely:
\begin{equation*}
U(T-T_0, T_0) e_{n}^v =e^{-(T-2T_0)(n+1)^2}  e_{n+1}^u,\ \
U(T-T_0, T_0) e_{n+1}^u =-e^{-(T-2T_0)(n+1)^2}  e_{n}^v.
\end{equation*}
Taking the composition of maps $U(T_0, 0)$, $U(T-T_0, T_0)$ and $U(T, T-T_0)$, we have
\begin{equation}\label{v_0_1}
U(T, 0) e_{n}^v =e^{-T_0n^2 - (2n+1)\int_0^{T_0}\theta_2(y(t))dt }e^{-(T-2T_0)(n+1)^2}e^{-T_0(n+1)^2}  e_{n+1}^u,
\end{equation}
and
\begin{equation}\label{u_0_1}
U(T, 0) e_{n+1}^u =-e^{-T_0(n+1)^2} e^{-(T-2T_0)(n+1)^2}e^{-T_0n^2 - (2n+1)\int_{T - T_0}^{T}\theta_2(y(t))dt }  e_{n}^v.
\end{equation}
It is remained to consider equations \eqref{exp_au} on the half-period $[T, 2T]$:
\begin{equation}\label{exp_au_2}
\begin{cases}
\pt v = \Dx v  - \varepsilon u \theta_1(-y),\\
\pt u = \Dx u + \varepsilon v \theta_1(-y),
\end{cases}
\end{equation}
the situation here is more or less similar to the case of interval $[0,T]$. Indeed,
due to the specific form of the cut-off function $\theta_1(y)$ and periodic function $y(t)$, on the time intervals $[T, T + T_0]$ and $[2T - T_0, 2T]$ the equations are decoupled:
\begin{equation}
\begin{cases}
\pt v = \Dx v,\\
\pt u = \Dx u.
\end{cases}
\end{equation}
Therefore
\begin{equation*}
U(T + T_0,T) e_{n}^v =e^{-T_0n^2 }  e_{n}^v, \ \
U(T +T_0,0) e_{n}^u =e^{-T_0n^2}  e_{n}^u,
\end{equation*}
and
\begin{equation*}
U(2T,2T - T_0) e_{n}^v =e^{-T_0n^2 }  e_{n}^v, \ \
U(2T,2T - T_0) e_{n}^u =e^{-T_0n^2}  e_{n}^u.
\end{equation*}
Equations \eqref{exp_au_2} on an interval $[T+T_0, 2T- T_0]$ have a form
\begin{equation}
\begin{cases}
\pt v = \Dx v  - \varepsilon u \theta_1(-y),\\
\pt u = \Dx u + \varepsilon v \theta_1(-y),
\end{cases},
\end{equation}
and we see that in Fourier coordinates $v_n$ is coupled with $u_n$ in comparison to the case of interval $[T_0, T-T_0]$, where $v_n$ was coupled with $u_{n+1}$. Namely,

\begin{equation}
\begin{cases}
\frac{d}{dt}v_n = - n^2 v_n - \varepsilon u_n\theta_1(-y),\\
\frac{d}{dt}u_n = - n^2 u_n + \varepsilon v_n\theta_1(-y).
\end{cases}
\end{equation}
As before, we introduce the polar coordinates:
\begin{equation}
v_n + i u_n = r_n e^{i\psi_n},
\end{equation}
and obtain the following equations on radial coordinate $r_n$ and angular coordinate $\psi_n$:
\begin{equation}
\frac{d}{dt}r_n = - n^2 r_n, \ \ \ \frac{d}{dt}\psi_n =  \varepsilon \theta_1(-y).
\end{equation}
Substituting $\varepsilon$ from \eqref{eps} and using symmetry of $y(t)$ we see that phase $\psi_n$  changes on $\pi/2$ on the interval $[T + T_0, 2T - T_0]$. Thus the map $U(2T - T_0, T + T_0)$ restricted on the $\operatorname{span}\left\{  e_{n}^v,  e_{n}^u \right\}  $ is a composition of the rotation and the contraction:
\begin{equation*}
U(2T -T_0,T + T_0) e_{n}^v =e^{-(T - 2T_0)n^2 }  e_{n}^u,
\end{equation*}
and
\begin{equation*}
U(2T - T_0,T + T_0) e_{n}^u =-e^{-(T - 2T_0)n^2}  e_{n}^v.
\end{equation*}
Therefore the composition of maps $U(T + T_0, T)$, $U(2T - T_0, T + T_0)$ and $U(2T, 2T -T_0)$ gives us
\begin{equation}\label{v_0_2}
U(2T, T ) e_{n}^v =e^{-T_0n^2 }e^{-(T - 2T_0)n^2}e^{-T_0n^2 } e_{n}^u
\end{equation}
and
\begin{equation}\label{u_0_2}
U(2T, T ) e_{n}^u =-e^{-T_0n^2 }e^{-(T - 2T_0)n^2}e^{-T_0n^2 }  e_{n}^v.
\end{equation}
 Formulas \eqref{v_0_1}, \eqref{u_0_1}, \eqref{v_0_2} and \eqref{u_0_2} guarantee that the Poincare map $P = U(2T,T)\circ U(T, 0)$ satisfies properties \eqref{P} with
\begin{multline}
\mu_n = -e^{-2T(n+1)^2 - (2n +1)\int_0^{T_0}(\theta_2(y(t))-1)dt} \text{ and }\\
\nu_n = -e^{-2Tn^2 -(2n +1)T - (2n +1)\int_0^{T_0}(\theta_2(y(t))-1)dt}.
\end{multline}
Thus, the theorem is proved.
\end{proof}
\begin{remark}\label{Rem4.cru} It follows from the explicit form of the functions $f(t,x)$ and $g(t,x)$ that they can be written in the form
$$
g(t,x)=\bar g(y(t),e^{ix}), \ \ \bar f(t,x)=f(y(t),e^{ix})
$$
for some $C^\infty$-functions $\bar f$ and $\bar g$. Moreover fixing $y(t)=\sin\frac {\pi t}T$, we may achieve that
\begin{equation}\label{4.fg}
g(t,x)=\bar g(\sin\frac{\pi t}T,\cos x,\sin x), \ \ f(t,x)=\bar f(\sin\frac{\pi t}T,\cos x,\sin x).
\end{equation}
Moreover, as also follows from the construction, the functions $\bar f$ and $\bar g$ are linear with respect to $\sin x$ and $\cos x$.
\end{remark}
We turn now to the nonlinear case. For the reader convenience, we start with discussing some known facts on Lipschitz manifolds, finite-dimensional reduction and attractors, see \cites{rom-th,rom-th1,28,Zel2} for more details.

\begin{definition}\label{Def3.man} A set $\mathcal M$ is a Lipschitz submanifold of dimension $N$ of
 a Hilbert space $\Phi$ if it can be presented locally as a graph of a Lipschitz continuous function. In other words, for any $u_0\in\mathcal M$, there exist $\eb=\eb(u_0)>0$, the open neighborhood $\mathcal V_{u_0}$ of $u_0$ in $\Phi$,   the projector $\mathcal P_{u_0}\in\mathcal L(\Phi,\Phi)$ of rank $N$ and a Lipschitz continuous map $M_{u_0}:\mathcal P_{u_0}\Phi\to (1-\mathcal P_{u_0})\Phi$ such that
 \begin{equation}\label{4.local}
 \mathcal M\cap \mathcal V_{u_0}=\{u_++M_{u_0}(u_+),\ \ u_+\in  B(\eb,\mathcal P_{u_0}u_0,P_{u_0}\Phi)\}.
 \end{equation}
 In particular, this means that
 \begin{equation}\label{4.lip}
 \|u-v\|_{\Phi}\le L_{u_0}\|\mathcal P_{u_0}(u-v)\|_{\Phi}
 \end{equation}
 for all $u,v\in \mathcal V_{u_0}$ and some constant $L_{u_0}$ which is independent of $u$ and $v$.
\end{definition}
\begin{remark}\label{Rem3.alt} Note that there is an alternative definition of a Lipschitz manifold which is also widely used in the literature. Namely, $\mathcal M$ is a Lipschitz manifold in $\Phi$ of dimension $N$ if for every $u_0\in\mathcal M$ there exists a neighborhood $\mathcal V_{u_0}$ of $u_0$ in $\Phi$, the number $\eb=\eb(u_0)>0$ and a bi-Lipschitz homeomorphism
$$
M: B(\eb,0,\R^N)\to\mathcal M\cap\mathcal V_{u_0}.
$$
As elementary examples show, these two definitions are {\it not} equivalent (actually, the second one is weaker than the first one), so the choice of the proper definition becomes important. Our choice of the first definition is motivated by the following two reasons: 1) it naturally generalizes the concept of a submanifold from the smooth to Lipschitz cases and, to the best of our knowledge, all known constructions of inertial manifolds automatically give the structure \eqref{4.local}; 2) we do not know whether or not the key statement about the finite-dimensionality of the dynamics on the attractor embedded into the finite-dimensional Lipschitz manifold holds without the assumption \eqref{4.lip}, see below.
\end{remark}
Let now $\mathcal A\subset \Phi$ be an attractor of the dissipative semigroup $S(t)$ generated by an abstract semilinear parabolic equation \eqref{0.abs}.
\begin{definition}\label{Def3.fin} We say that the dynamics generated by $S(t)$ on the attractor possesses a Lipschitz continuous inertial form if the following conditions are satisfied:
\par
1) There exist $N>0$ and an injective Lipschitz map $I:\mathcal A\to\R^N$ such that $I^{-1}:\bar{\mathcal A}:=I(\mathcal A)\to\mathcal A$ is also Lipschitz continuous.
\par
2) There exists a Lipschitz continuous vector field $G$ on $\bar{\mathcal A}\subset\R^N$ such that the projected semigroup $\bar S(t):=I\circ S(t)\circ I^{-1}$ on $\bar{\mathcal A}$ is a solution semigroup of the following ODEs:
\begin{equation}\label{3.ODE}
\frac d{dt}U=G(U),\ \ U\big|_{t=0}=I(u_0),\ \ u_0\in\mathcal A.
\end{equation}
This system of ODEs is referred  then as an initial form associated with \eqref{0.abs}.
\end{definition}
We give below only several known facts on such inertial forms which are crucial for our purposes, more details can be found in \cites{rom-th,rom-th1}.
\begin{proposition}\label{Prop3.gr} Under the above assumptions the Lipschitz continuous inertial form exists if and only if the semigroup $S(t)$ restricted to the global attractor $\mathcal A$ can be extended for negative times to a Lipschitz continuous group $\{S(t),\ t\in\R\}$. Moreover, then the spectral projector $P_N$ can be used as a map $I$ for sufficiently large $N$.
\end{proposition}
Indeed, in one side the statement is obvious since Lipschitz vector field in $\R^N$ generate Lipschitz continuous solution groups. In the opposite side it is a bit more delicate and require some efforts, see \cite{rom-th1}.
\begin{proposition}\label{Prop3.emb} Under the above assumptions the Lipschitz continuous inertial form exists if and only if there exists a finite-dimensional Lipschitz submanifold (not necessarily invariant) containing the global attractor $\mathcal A$.
\end{proposition}
This statement is  proved in \cite{rom-th}[Theorem 1.5] (actually, the existence of an inertial form is verified there under the extra assumption that the manifold is $C^1$-smooth, but this fact is used {\it only} in order to obtain estimate \eqref{4.lip} which is incorporated in our case into the definition of the Lipschitz submanifold, see Remark \ref{Rem3.alt}).
\par
We are now ready to state and prove the second main result of the paper on the non-existence of IMs for systems  of RDAs.

\begin{theorem}\label{Th4.main} There exists an example of the RDA system \eqref{eq_1D} with the number of equations $m=8$ and the nonlinearities $f$ and $g$ satisfying the assumptions of Theorem \ref{dis} such that the associated global attractor is not a subset of any finite-dimensional Lipschitz continuous submanifold of the phase space $\Phi$. In particular, this equation does not possess an inertial manifold.
\end{theorem}
\begin{proof} Our strategy is the following: to verify the non-existence, we will find two trajectories ${\rm u}_1(t)$ and ${\rm u}_2(t)$ belonging to the attractor $\mathcal A$ such that
\begin{equation}\label{3.fast}
\|{\rm u}_1(t)-{\rm u}_2(t)\|_{\Phi}\le Ce^{-\gamma t^3}.
\end{equation}
The existence of such trajectories does not allow to extend the solution semigroup $S(t)$ on the attractor to a Lipschitz continuous group and, thanks to Proposition \ref{Prop3.gr}, the associated Lipschitz inertial form does not exist. Then, applying Proposition \ref{Prop3.emb}, we see that the embedding of the attractor to any Lipschitz submanifold is also impossible. Thus, it only remains to find the trajectories satisfying \eqref{3.fast}.
\par
We construct the desired example based on the counterexample given in Theorem \ref{TH2} using \eqref{4.fg} and interpreting the functions $y(t)=\sin\frac{\pi t}T$, $y_1(x)=e^{ix}$ as particular solutions of extra RDA equations. However, to fulfill the other assumptions, we need to modify slightly equations \eqref{au}. Namely, let us introduce a cut-off function  $\phi(\xi)$ such that $\phi(\xi) = 1$ for $|\xi|\le 1/4$ and $\phi(\xi) = 0$ for $|\xi|\ge 1/2$ and consider the RDA system
\begin{equation}\label{3.int}
\Dt{\rm u}=\Dx {\rm u}+\phi(|{\rm u}|^2)(f(t,x)\Nx {\rm u}+g(t,x){\rm u})+(1-\phi(|{\rm u}|^2))({\rm u}-{\rm u}|\rm u|^2),
\end{equation}
where $f$ and $g$ are exactly the same as in Theorem \ref{TH2}. Then, on the one hand, this system remains linear near the origin ${\rm u}=0$, so ${\rm u}=0$ is a super exponentially attracting equilibrium. On the other hand, the presence of the nonlinearity of a Ginzburg-Landau type makes the system dissipative and produces extra equilibria filling the sphere $|{\rm u}|=1$.
\par
Let $P:=U(2T,0):\Phi\to\Phi$ be the period map generated by the nonlinear equation \eqref{3.int}. Then, since \eqref{3.int} is time-periodic, $U(2nT,0)=P^n$ and the dynamics of \eqref{3.int} is determined by the discrete semigroup $S_{per}(n):=P^n$ generated by the iterations of the map $P$. In particular, as not difficult to see arguing as in Theorem \ref{dis}, this semigroup possesses a global attractor $\mathcal A_{tr}$ which is a compact connected set in the phase space $\Phi$. Obviously, the attractor contains all equilibria
$$
\{0\}\cup\{{\rm u}\in\R^4,\ |{\rm u}|=1\}\subset \mathcal A_{per}.
$$
Furthermore, since $0$ is locally asymptotically stable and the attractor is connected, there exists a non-trivial complete bounded trajectory ${\bar u}_2(t)$, $t\in\R$ such that ${\rm u}_2(t)\to0$ as $t\to\infty$. Finally, since \eqref{3.int} coincides with \eqref{au} in the neighbourhood of zero, from Theorem \ref{TH2} we conclude that
\begin{equation}\label{3.super}
\|{\rm u}_1(t)-{\rm u}_2(t)\|_\Phi\le Ce^{-t^3},\ \ {\rm u_1}(t)\equiv 0
\end{equation}
(here we have implicitly used the smoothing property in order to obtain the attraction in the norm of $\Phi$).
\par
We are now ready to embed system \eqref{3.int} into a large autonomous and spatially homogeneous system of RDA equations. To this end, we note that functions $y(t,x):=e^{\pi it/T}$ and $z(t,x)=e^{ix}$
solve the semilinear heat equations
\begin{equation}
\Dt y=\Dx y+\frac{\pi i}Ty+y(1-|y|^2),\ \ \Dt z=\Dx z+z(2-|z|^2)
\end{equation}
respectively, so we may introduce the extended system
\begin{equation}\label{3.fin}
\begin{cases}
\Dt y=\Dx y+\frac{\pi i}Ty+y(1-|y|^2),\\
\Dt z=\Dx z+z(2-|z|^2),\\
\Dt{\rm u}=\Dx {\rm u}+\phi(|{\rm u}|^2)(f(\Imm y,\Ree z,\Imm z)\Nx {\rm u}+\\\text{\phantom{eggogeggogeggogeggog}}+g(\Imm y,\Ree z,\Imm z){\rm u})+(1-\phi(|{\rm u}|^2))({\rm u}-{\rm u}|\rm u|^2).
\end{cases}
\end{equation}
The number of equations in this system is $2+2+4=8$.
Then, as not difficult to see, the system \eqref{3.fin} of RDA equations is dissipative and possesses a global attractor $\Cal A$ in the phase space $\Phi$. On the other hand, by the construction, the trajectories ${\rm U}_1(t):=(e^{\pi i T},e^{ix},{\rm u}_1(t))$ and ${\rm U}_2(t):=(e^{\pi i/T},e^{ix},{\rm u}_2(t))$, $t\in\R$, solve these equations. Moreover, since these are complete bounded trajectories, they belong to the global attractor $\mathcal A$:
\begin{equation}
{\rm U}_1(t),{\rm U}_2(t)\in\mathcal A,\ \ t\in\R
\end{equation}
Finally,
$$
\|{\rm U}_1(t)-{\rm U}_2(t)\|_{\Phi}=\|{\rm u}_1(t)-{\rm u}_2(t)\|_{\Phi}\le Ce^{-\gamma t^3},
$$
so the Lipschitz extension of $S(t)$ on the attractor for negative times does not exists and the attractor $\mathcal A$ is not a subset of any Lipschitz finite-dimensional submanifold of $\Phi$.
\par
It only remains to note that, although the constructed nonlinearities formally do not satisfy the assumptions of Theorem \ref{dis} since they do not have finite supports, but this can be easily corrected by cutting of the nonlinearities outside of a large ball. Thus, the theorem is proved.
\end{proof}
\begin{remark}\label{Rem3.fin} The obtained counterexample excludes the embeddings of the global attractor in Lipschitz submanifolds, but does not forbid the existence of log-Lipschitz inertial forms and related embeddings to log-Lipschitz manifolds which are of big current interest, see \cite{28} and references therein. However, the constructed counterexample to the Floquet theory is the key point of the proof of non-existence of such forms given in \cite{EKZ} for the case of abstract parabolic equations, so we expect that the analogous counterexample could be extended in a straightforward way to the case of RDA equations. Since the construction given in \cite{EKZ} is rather technical we decided not to present it here.
\end{remark}

\section{Concluding remarks}\label{s5}

In this concluding section, we briefly discuss possible generalizations of the obtained results. We start with particular cases of systems where the IM still exists.

\subsection{Vector case and existence of IMs} As the constructed in Theorem \ref{Th4.main} counterexample shows, we cannot expect the existence of IMs under general assumptions on the nonlinearity $f$. However, this a priori does not exclude the existence of IMs if the matrix $f$ has some specific structure. In particular, it will be so if the matrix $f(u)$ has a diagonal structure with only one non-zero entry on the diagonal:
\begin{equation}\label{5.diag}
f(u)=\operatorname{diag}(f_1(u),\cdots,f_m(u))
\end{equation}
and
\begin{equation}\label{5.one}
f_i(u)=\psi(u)\delta_{ij}
\end{equation}
for some $j\in\{1,\cdots, m\}$. It worth emphasizing that, in contrast to the previous section all functions are real-valued here.
\par
\begin{proposition}\label{Prop5.single} Let the assumptions of Theorem  \ref{dis} holds and let, in addition, the nonlinearity $f$ satisfy \eqref{5.diag} and \eqref{5.one}. Then problem \eqref{eq_1D} possesses an IM in the phase space $\Phi$.
\end{proposition}
Indeed, in this case we need to transform only one component of $u=(u_1,\cdots,u_m)$ via $u_j(t,x)=a(t,x)w_j(t,x)$ and we will have a scalar equation on the factor $a$ which can be solved exactly as in Section \ref{s1}. So, the IM can be constructed exactly as in the case of a scalar equation considered above.
\par
Analogously to the case of Dirichlet or Neumann boundary conditions, this simple observation allows us to treat the case of scalar quasilinear equation
\begin{equation}\label{5.quasi}
\Dt u=\Dx u+f(u,\Nx u)
\end{equation}
with periodic boundary conditions. Indeed, differentiating this equation by $x$ and denoting $v=\Nx u$, we end up with a system of RDA equations
\begin{equation}\label{5.full}
\begin{cases}
\Dt u=\Dx u +f(u,v),\\
\Dt v=\Dx v+f'_u(u,v)v+f'_{v}(u,v)\Nx v
\end{cases}
\end{equation}
which satisfies assumptions of Proposition \ref{Prop5.single} and, therefore, possesses an IM.

\begin{remark}\label{Rem5.un1} Note that in the constructed counterexample only {\it two} components of the matrix $f(u)$ are non-zero and this is enough to destroy the existence of IMs, so the assumptions of Proposition \ref{Prop5.single} are in a sense sharp. Note also that this nonlinearity will satisfy assumptions \eqref{5.diag} and \eqref{5.one} if we allow the components of $u$ to be complex-valued, so the assumption that $u$ is real-valued is crucial for the validity of Proposition \ref{Prop5.single}.
\end{remark}
\subsection{Mixed Dirichlet-Neumann boundary conditions} As shown in the first part of this work (see \cite{KZI}), an IM exists for systems of RDAs in the case of Dirichlet boundary conditions as well as for Neumann boundary conditions. Surprisingly, it may be not the case if some components of the vector $u=(u_1,\cdots, u_m)$ are endowed by the Dirichlet and the rest by the Neumann boundary conditions. To see this we start with the counterexample constructed in Theorem \ref{TH2} for periodic boundary conditions  which we write here as
\begin{equation}\label{5.U}
\Dt{\rm U}=\Dx {\rm U}+f({\rm U})\Nx {\rm U}+g({\rm U})
\end{equation}
and introduce the functions
$$
{\rm U}_{alt}(t,x):={\rm U}(t,x)-{\rm U}(t,2\pi-x),\ \ {\rm U}_{sym}={\rm U}(t,x)+{\rm U}(t,2\pi-x).
$$
Then, obviously, the functions ${\rm U}_{alt}$ and ${\rm U}_{sym}$ satisfy the Dirichlet and Neumann boundary conditions respectively and
$$
{\rm U}(t)=\frac12({\rm U}_{alt}(t)+{\rm U}_{sym}(t)).
$$
On the other hand,
$$
\Dt {\rm U}(t,2\pi-x)=\Dx{\rm U}(t,2\pi-x)-f({\rm U}(t,2\pi-x))\Nx{\rm U}(t,2\pi-x)+g({\rm U}(t,2\pi-x)).
$$
Taking a sum and a difference of this equation and equation \eqref{5.U}, we end up with the following equations
\begin{multline}
\Dt {\rm U}_{sym}=\Dx {\rm U}_{sym}+\\+\frac12\(f(({\rm U}_{alt}+{\rm U}_{sym})/2)\Nx ({\rm U}_{alt}+{\rm U}_{sym})-f(({\rm U}_{sym}-{\rm U}_{alt})/2)\Nx ({\rm U}_{sym}-{\rm U}_{alt})\)+\\+g(({\rm U}_{alt}+{\rm U}_{sym})/2)+g(({\rm U}_{sym}-{\rm U}_{alt})/2)
\end{multline}
and
\begin{multline}
\Dt {\rm U}_{alt}=\Dx {\rm U}_{alt}+\\+\frac12\(f(({\rm U}_{alt}+{\rm U}_{sym})/2)\Nx ({\rm U}_{alt}+{\rm U}_{sym})+f(({\rm U}_{sym}-{\rm U}_{alt})/2)\Nx ({\rm U}_{sym}-{\rm U}_{alt})\)+\\+g(({\rm U}_{alt}+{\rm U}_{sym})/2)-g(({\rm U}_{sym}-{\rm U}_{alt})/2)
\end{multline}
The obtained system is a system of $16$ RDA equations first $8$ of which are endowed by the Neumann boundary conditions and the second $8$ equations have Dirichlet boundary conditions. Since the attractor of this system contains the attractor of the system \eqref{5.U}, it also does not possess an inertial manifold.
\par
This example confirms once more that the existence or non-existence of IMs for the systems of RDA equations strongly depends on the choice of boundary conditions.

\bibliography{ReferencesI}

\end{document}